\documentclass{article}
\usepackage{amssymb,amsmath,amsfonts,amsthm,latexsym}
\usepackage{multicol}
\usepackage{graphicx}

\newtheorem*{theo*}{Theorem}
\newtheorem{theo}{Theorem}
\newtheorem{coro}[theo]{Corollary}
\newtheorem{prop}[theo]{Proposition}
\newtheorem{lemm}[theo]{Lemma}
\newtheorem*{lemm*}{Lemma}
\newtheorem{conj}{Conjecture}

\newtheorem*{clai*}{Claim}

\theoremstyle{definition}
\newtheorem{rema}{Remark}
\newtheorem{defi}[rema]{Definition}

   \def\PP{{\mathbb P}}
 \def\RR{{\mathbb R}}  
   
 \def\ZZ{{\mathbb Z}}

\def\La{\Lambda}

\def\cA{{\cal A}}    
   \def\cN{{\cal N}} 
   \def\cO{{\cal O}} \def\cU{{\cal U}}
    \def\cV{{\cal V}}
    
   \def\cR{{\cal R}} \def\cX{{\cal X}}

\def\set#1{\left\{\, #1 \,\right\}}
\def\tq{\,\mid\,}
\def\abs #1{\vert \,#1\, \vert\,}
\def\norm #1{\Vert \,#1\, \Vert\,}

\title{Star flows with singularities of different indices}
\author{Adriana da Luz}

\begin{document}

\maketitle


\begin{abstract}

It is known that a generic star vector field $X$ on a $3$ or $4$-dimensional
manifold is such that its chain recurrence classes are either  hyperbolic, or singular hyperbolic (\cite{MPP} and \cite{GSW}).
 Palis conjectured that every vector field must be approximated  either by singular hyperbolic vector fields or by  vector fields with homoclinic tangencies or heterodimensional cycles (associated to periodic orbits).
We give a counter example in dimension $5$ (and higher).

We present here an open set of   star  vector fields  on a $5$ -dimensional manifold for which two singular points
with different indices belong (robustly)
to the same chain recurrence class.
This prevents the class to be singular hyperbolic, showing that the results in \cite{MPP} can not be extended to higher dimensions and thus contradicting the conjecture by Palis.

\end{abstract}


\newpage
\tableofcontents
\newpage
\section{Introduction}
The theory of dynamical systems has its origins in  the work of Poincar\'e while  studying the time continuous dynamical systems,(vector fields and their flows) associated with celestial mechanics.
Later on, we began to study discrete time dynamical systems, that is diffeomorphisms.
Both theories seem to be very related. Even more, there was an idea shared by many authors that is:

\emph{The dynamics of a non singular vector fields  in
dimension $n$ should look like the one of a diffeomorphism in dimension $n-1$}( see for instance \cite{Sm1})

However, this idea of translating results from one setting to the other does not always work quite as straightforwardly as expected. One of the possible situations where this happens is when we are dealing with vector fields with singularities (zeros of the vector field).
\vskip 2mm
The coexistence of singularities and regular orbits in indecomposable parts of the dynamics has lead to the fact that there are several areas in which the theory for vector fields is somewhat behind the theory for  diffeomorphisms.
 \vskip 2mm
  The aim of this paper will be to build and example showing new ways in which singularities do introduce extra difficulties in the vector field scenario.
\subsection{Diffeomorphisms, hyperbolicity and the Palis conjecture}


In the 60\' s, and after the famous work of Smale in \cite{Sm}, begun the study of a particular open set of diffeomorphism, that are such that the tangent space over some compact invariant subset,  splits into two invariant bundles (under the tangent dynamics). In the tangent dynamics vectors are uniformly expanded in one bundle and uniformly contracted in the other. This systems where called hyperbolic systems.
This class of diffeomorphisms was the focus of several works which arrived to a very complete description of their dynamics

The study of hyperbolic systems played a central roll in differentiable  dynamical systems. One of the reasons for this is that they are stable in the following sense: Any hyperbolic system has a $C^r$ neighborhood such that all diffeomorphisms in that neighborhood share all the same dynamical properties (all diffeomorphisms in the neighborhood are conjugated).

In \cite{PaSm} Palis and Smale conjectured that hyperbolicity characterizes stability, (it was conjectured for the $C^r$ topology but it was only solved for $r=1$). The sufficient condition was proven by \cite{R1} and  \cite{R2}. The necessary condition was proven by Ma\~{n}e in 1988 \cite{Ma2} and by Hayashi \cite{H} 1992.

Less informally,
 given a compact invariant set $K$ of a diffeomorphism $f$ we say that $f$ is \emph{hyperbolic or uniformly hyperbolic} on $K$, if \begin{itemize}\item  there is a continuous, invariant  splitting of the tangent space, in two spaces: $T_xM=E_x^s\oplus E_x^u$ \item the vectors are uniformly contracted in $E^s$ \item the vectors are uniformly expanded in $E^u$.\end{itemize}
  We can define an analogous notion for vector fields.

When we say that a diffeomorphism $f$ (or a flow) is hyperbolic we do not mean $K $ is the hole manifold in this text.  Rather  that $K$ is a compact invariant subset of the manifold called the chain recurrent set. 
This set holds the most relevant dynamical properties, we define it as follows.
 \begin{itemize}
 \item A point $x$ is \emph{chain recurrent} if, for any $\varepsilon>0$, there is an \emph{$\varepsilon$-pseudo orbit} from
 $x$ to $x$, that is,  a sequence $x=x_0, x_1\dots,x_k=x$, $k>0$ with $d(x_i, f(x_{i-1}))<\varepsilon$, for $i\in \{1,\dots, k\}$.
Equivalently $x$ is chain recurrent if  for any attracting region $U$ (an open set such that $f(U)\subset U$ ), the orbit of $x$ is either disjoint from $U$ or contained in it.
\item The set of chain recurrent points is the chain recurrent set that we note as  $\cR(f)$. 
 \item Two points $x,y$ in $\cR(f)$ are in the same \emph{chain recurrence class} if  for any $\varepsilon>0$, there are $\varepsilon$-pseudo orbits from $x$ to $y$ and from $y$ to $x$.
\end{itemize}

In this way we say that $f$ is hyperbolic if it is hyperbolic in $\cR(f)$.

 It is shown by Conley in \cite{Co} that this chain classes play the role of fundamental pieces of the dynamics, and the rest of the orbits, simply go from one of this pieces to the other.

We say a set $\La$ is maximal invariant in $U$ for a diffeomorphism $f$ if
$$\La=\bigcap_{n\in \ZZ} f^n(U)\,.$$

 One says that a system is \emph{star} if  all periodic orbits are hyperbolic in a robust fashion: every periodic orbit of every $C^1$-close system is hyperbolic.
As a step to proving the stability  conjecture, one proves that for a diffeomorphism, to be star  is equivalent to be hyperbolic
(an important step is done in \cite{Ma}  and has been completed in \cite{H}\label{star}).

However  hyperbolicity does not describe all systems  or even a dense set of the diffeomorphisms in the $C^1$ topology (see for instance Abraham Smale ' s example in \cite{ASm}). 

After this, Palis gave a series of conjectures with hopes of describing the dynamics of at least a dense set of diffeomorphisms far from hyperbolicity.

\begin{conj}[Palis density conjecture]
There is a dense open subset $\cO =\cO_1\cup\cO_2$ of $Diff^1(M)$ such
that  $f\in\cO_1$ satisfies the Axiom A without cycle (what we call here being hyperbolic), and there is a dense subset $D\subset \cO_2$ such that $ f\in D$
admits an hetero dimensional cycle or a homoclinic tangency.
\end{conj}
This conjecture has been proved for surface diffeomorphisms by Pujals and Sambarino in \cite{PS}. In
higher dimensions, two important steps in the positive direction have been done in \cite{C1} and \cite{CP}

In order to understand better the systems that are not hyperbolic,  when a  systems has some property that persist under small perturbation,  we aim to find (weaker)
structures that limit the effect of the small perturbations.

The weakest of this defined structures for diffeomorphisms was introduced by Ma\~ n\'e and Liao and it is
called \emph{dominated splitting}:
 \begin{defi} Let $f\colon M\to M$ be a diffeomorphism of a Riemannian manifold $M$  and $K\subset M$ a compact invariant set of $f$.
 A splitting $T_x M=E(x)\oplus F(x)$, for $x\in K$, is called dominated if
 \begin{itemize}
 \item $dim(E(x))$ is independent of $x\in K$ and this dimension is called the $s$-index of the splitting;
  \item it is $Df$-invariant:  $E(f(x))=Df(E(x))$ and $F(f(x))=Df(F(x))$ for every $x\in K$;
  \item there is $n>0$ so that for every $x$ in $K$ and every  unit vectors $u\in E(x)$ and $v\in F(x)$ one has
  $$\|Df^n(u)\|\leq \frac 12\|Df^n(v)\|.$$
 \end{itemize}
One denotes $TM|_K=E\oplus_{_<}F$ the dominated splitting.

 \end{defi}


Some examples of the relation between this structures and the robustness of the dynamical properties can be found in
the long sequence of papers, starting with the work \cite{Ma} of Ma\~n\'e, and then \cite{DPU,BDP, BDV} and the most complete result in this spirit, in \cite {BB}. They show that the dominated splittings is the unique obstruction
for mixing the Lyapunov exponents of periodic orbits (and therefore not having any robust dynamical property ), by $C^1$-small perturbation of the diffeomorphism.

\subsection{Singular flows}

However, as noted before, vector fields can have   singularities (zero of the vector field). This makes the translation of this results to the flow setting, more complicated. In fact, the $C^1$ stability conjecture for flows was proven by Hayashi (\cite{H2}) almost ten years after the result for diffeomorphisms.

\vspace{2mm}
\emph{The hyperbolic splitting of a singularity and the one over a periodic orbit are different and a priori not compatible.}
\vspace{2mm}

This is due to the fact that the vector associated  to a point  of a regular orbit, by the vector field, spans a space in the tangent space that does not contract or expand vectors. As opposed to this, singularities are  zeros of the vector field , and therefore there is no space spanned by the vector field. 

We cannot avoid this problem, since the singularities might be accumulated, in a robust way, by regular chain  recurrent orbits.

The first example with this behavior has been indicated by Lorenz in \cite{Lo} under  numerical evidences.  Then \cite{GuWi} constructs
a $C^1$-open set of vector fields in a $3$-manifold, having a topological transitive attractor containing  periodic orbits (that are all hyperbolic) and one
singularity.  The examples in \cite{GuWi} are known as the geometric Lorenz attractor.

For non singular star flows in \cite{GW} Gan and Wen show that star flows away from singularities are in fact hyperbolic but 

\begin{center}\emph{the Lorenz attractor is also an example of a robustly non-hyperbolic  star flow,  showing that the result in \cite{H} is not true anymore for flows.}\end{center}
\vspace{0,02cm} 
In dimension $3$ the difficulties introduced by the robust coexistence of singularities and periodic orbits are now almost fully understood.
In particular, Morales, Pac\'ifico and Pujals (see \cite{MPP} ) defined the notion of \emph{singular hyperbolicity}, which requires that the chain recurrence classes admit a dominated splitting
in two bundles, one being uniformly contracted (resp. expanded) and the other being volume expanding (resp. volume contracting).

 They prove that, for  a $C^1$-generic star flows on $3$-manifolds,
every chain recurrence class is singular hyperbolic.
In \cite{BaMo} the authors built a star flow on a $3$-manifold having a chain recurrence class which is not singular hyperbolic, showing that the previous result cannot be improved.

In the context of flows, Palis conjecture is formulated as follows:
\begin{conj}[Palis \cite{Pa} (conjecture 5)]\label{conjw}
In any dimension, every flow(vector field)can be $C^r$ approximated by a hyperbolic one or by one displaying a homoclinic tangency or a singular cycle or a heterodimensional cycle.
\end{conj}
This conjecture was proven for the $C^1$ topology and in a three dimensional manifold by Arroyo and Rodriguez-Hertz in \cite{AH}

Since all information regarding coexistence of singularities with periodic orbits in Conjecture \ref{conjw} falls under the category of singular cycles, the same article gives a stronger version of this conjecture where the behavior of this singular cycles would also be described (see the line following conjecture 6)
\begin{conj}[Palis \cite{Pa} (conjecture 5 strong version)]\label{conj}
In any dimension, every flow (vector field) can be $C^r$ approximated by a singular hyperbolic one or by one displaying a homoclinic tangency or an heterodimensional cycle (associated to periodic points).
\end{conj}

In \cite{CY} Crovisier and Yang give a proof of this conjecture also for the $C^1$ topology and in a three dimensional manifold. Both of this conjectures remained open until now, and the aim of this paper is to give a counter example to Conjecture \ref{conj} in dimension bigger or equal to 5.

The notion of singular hyperbolicity defined by \cite{MPP} in dimension $3$ admits a straightforward generalization to higher dimensions:

Each chain recurrence class admits a dominated splitting
in two bundles, one being uniformly contracted (resp. expanded) and the other being  area expanding in any two dimensional subspace  (resp. area contracting).
Other possible generalizations of this notion can be found in \cite{MM}

If the chain recurrence set of a vector field $X$  can be covered by filtrating sets $U_i$ in which the maximal invariant set $\La_i$ is singular hyperbolic, then
$X$ is a star flow.  Conversely, \cite{GLW} and \cite{GWZ}  prove that this property characterizes the  generic star flows on  $4$-manifolds and for robustly transitive singular sets.

In \cite{GSW} the authors prove the singular hyperbolicity of generic star flows in any dimensions assuming
an extra property:  if two singularities are in the same chain recurrence class then
they must have the same \emph{$s$-index} (dimension of the stable manifold). Indeed, the  singular hyperbolicity implies directly this extra property.

 \vspace{0,2cm}
 \emph{The definition of singular hyperbolicity forbids  a priory the coexistence of singularities of different indexes in the same chain recurrence  class.}
 \vspace{0,2cm}

If one assumes that the  chain recurrence class that we are looking at is robustly transitive, then \cite{GWZ}  prove that the condition on the index of the singularities is always verified.
It was  unknown weather the condition on the index is still verified after removing the hypothesis of robust transitivity. The following conjecture appears in \cite{GWZ} and  \cite{GSW}  and was proven in   \cite{GSW}  for an open and dense set of vector fields in dimension 4.  
\begin{conj}\label{conj2}For every star vector field the chain recurrent set
is singular hyperbolic and consists of finitely many chain recurrent classes.
\end{conj}
From the example of $\cite{BaMo}$ it was expected that this conjecture would only hold in an open and dense set. 
This conjecture is very closely related to the conjecture by Palis ([\ref{conj}]): open and densely, star flows do not have homoclinic tangencies or heterodimensional cycles, therefore if Palis conjecture were true then this conjecture should be true as well.

One of the difficulties that persisted for understanding this questions in higher dimension is that
there are very little examples illustrating what are the possibilities. Let us mention \cite{BLY} which builds  a flow having a robustly chain recurrent attractor containing saddles of different indices.

The contents of this work aim to contribute in this direction. We give a negative answer to the conjecture by Palis [ Conjecture \ref{conj}] and to Conjecture [\ref{conj2}] in dimension grater or equal than 5. We build a robust counter example of a star flow  such that  no  known notion of  singular or sectional hyperbolicity would be satisfied, but it is away from homoclinic tangencies and heterodimensional cycles.
In fact, this example has a center space (a space that does not contract or expand) of dimension at least three and this implies that any partial hyperbolic structure on the tangent space that this example may satisfy does not imply the star condition. Therefore a characterization via hyperbolic structures of the tangent space for star flows is impossible even open and densely.

More over, it has all periodic orbits robustly hyperbolic and of the same hyperbolic index, but the finest dominated splitting (if any) would have a center space of dimension at least 3, where the Lyapunov exponents cannot be mixed by perturbation, which show that results like \cite{BDP} or \cite{BB} do not hold in a direct translation to flows.

This  star flow will be  in a $5$ dimensional manifold, and has two singularities of different
indices which belong to the same chain recurrence class, robustly.

\begin{theo}\label{t.example} Let $M$ be the manifold $S^3\times\RR\PP^2$. There is a $C^1$-open set $\cU$ of $\cX^1(M)$ so that every $X\in \cU$ is
such that there is an open set $U$,
\begin{itemize}
  \item  such that $X\in \cU$ is a star flow in $ U$.
  \item the maximal invariant set in $U$ is a chain recurrence class $C$
  \item  $C$ has  two singularities $\sigma_1$ and $\sigma_2$ such that the stable manifold of $\sigma_1$ is 3 dimensional and the stable manifold of $\sigma_2$ is 2 dimensional
  \item these singularities are such that  $\sigma_1$ and $\sigma_2$  belong to $\overline{Per(X)}$

\end{itemize}

\end{theo}




In order to prove that the example we construct is actually a star flow, we need some tool that allows us to detect the robust hyperbolicity of the periodic orbits without any information of the neighboring vector fields.
For this we define a hyperbolic structure that we call \emph{strong multisingular hyperbolicity} which is a particular case of the \emph{ multisingular hyperbolicity} defined in \cite{BdL} that is a sufficient condition to be a star flow.
We then construct our example so that it is strong multisingular hyperbolic, and therefore a star flow.

The hyperbolic structure we will define does not lie on the
tangent bundle, but in the normal bundle with the  linear Poincar\'e flow.   However the linear Poincar\'e flow is only defined far from the singularities, and therefore it cannot be used directly for understanding
our example.

In \cite{GLW},  the authors  define the notion of \emph{extended linear Poincar\'e flow} defined on some sort of blow-up of the singularities.
The notion of \emph{strong multisingular hyperbolicity} will be expressed as  the hyperbolicity of a  reparametrization of this extended linear Poincar\'e flow,
over a well chosen extension of the chain  recurrence set.

Since we do not have any information of the neighboring vector fields, we will need to extend the  linear Poincar\'e flow to some set that is interesting to us from the dynamical point of view, that varies upper semi continuously with the vector field, but that does not depend on knowing information from the neighborhood of the vector field. Therefore we will need a different notion as the one defined in \cite{GLW}. We will use the notion of central space as defined in \cite{BdL}.

\section{Basic definitions and preliminaries}
\subsection{Chain recurrence classes and filtrating neighbourhoods }

The following notions and theorems are due to Conley \cite{Co} and they can be found in several other references (for example \cite{AN}).
\begin{itemize}
  \item We say that pair of sequences $\set{x_i}_{0\leq i\leq k}$ and  $\set{t_i}_{0\leq i\leq k-1}$, $k\geq 1$,  is an \emph{ $\varepsilon$-pseudo orbit from $x_0$ to $x_k$} for a flow $\phi$,
  if for every $0\leq i \leq k-1$ one has
  $$ t_i\geq 1 \mbox{ and }d(x_{i+1},\phi^{t_i}(x_i))<\varepsilon.$$

  \item A compact invariant set $\Lambda$ is called \emph{chain transitive} if    for any $\varepsilon > 0$, and for any $x, y \in\Lambda$
there is an $\varepsilon$-pseudo orbit from $x$ to $y$.
  \item We say that $x, y \in M$ are chain related if,  for every $\varepsilon>0$, there are $\varepsilon$-pseudo orbits form $x$ to $y$ and from $y$ to $x$. This is an equivalence relation.
\item  We say that $x\in M$ is \emph{chain recurrent} if  for every $\varepsilon>0$, there is a non trivial $\varepsilon$-pseudo orbit from $x$ to $x$. We call the set of chain recurrent points, the\emph{ Chain recurrent set}
and we note it  $\cR(f)$. The equivalent classes of this equivalence relation are called \emph{ chain recurrence classes}.
\end{itemize}

\begin{defi}
\begin{itemize} \item An \emph{attracting region} (also called \emph{trapping region} by some authors) is a compact set $U$ so that $\phi^t(U)$ is contained in the interior
of $U$ for every $t>0$. The maximal invariant set in an attracting region is called an \emph{attracting  set}.  A repelling region is an attracting region for $-X$, and the maximal invariant set is called a repeller.

\item A \emph{filtrating region} is the intersection of an attracting region with a repelling region.

\item Let $C$ be a chain recurrence class of $M$ for the flow $\phi$.
A \emph{filtrating neighbourhood } of $C$ is a (compact) neighbourhood which is a filtrating region.
\end{itemize}
\end{defi}

The following is a corollary of the fundamental theorem of dynamical systems  \cite{Co}.
\begin{coro}\cite{Co}
Let $X$  be a $C^1$-vector field on a compact manifold $M$. Every chain class $C$  of $X$ admits a  basis of filtrating neighbourhoods, that is, every neighbourhood
of $C$ contains a filtrating neighbourhood of $C$.
\end{coro}

\begin{defi}
Let $C$ be a chain recurrence class of $M$ for the vector field $X$.
We say that $C$  is \emph{robustly chain transitive} if there exist a filtrating neighbourhood $U$ of $C$, and $C^1$ neighbourhood of $X$ called $\mathcal{U}$ such that for every $Y\in\mathcal{U}$, the maximal invariant set for $Y$ ($C_Y$) in $U$ is a unique  chain class.
\end{defi}

\begin{defi}
Let $C$ be a robustly chain transitive class of $M$ for the vector field $X$.
We say that $C$  is robustly transitive if there is a $C^1$ neighbourhood of $X$ called $\mathcal{U}$ such that for every $Y\in\mathcal{U}$, there is an orbit for $Y$ which is dense in $C_Y$.
\end{defi}

\subsection{Linear cocycles}
Let  $\phi=\{\phi^t\}_{t\in\RR}$ be a topological flow on a compact metric space  $K$.

Consider as well:
\begin{itemize}
\item A $d$ dimensional  linear bundle $E$ over $K$ with $\pi:E\to K$
\item A continuous map $A_t:(x,t)\in K\times \RR \mapsto GL(E_x,E_{\phi^t(x)})$  that satisfies the following \emph{cocycle relation }:
for any $x\in K$  and $t,s\in\RR$ one has:  $$A_{t+s}(x)=A_t(\phi^s(x))A_s(x)$$

\end{itemize}
We define a \emph{linear cocycle over   $(K,\phi)$}   as  the associated morphism   $A^t\colon K\times \RR\to K$
defined by $$A^t(x, v) = (\phi^t(x), A_t(x)v)\,.$$

Note that $\cA=\{A^t\}_{t\in\RR}$ is a flow on the space $E$ which projects on $\phi^t$.
$$\begin{array}[c]{ccc}
E &\stackrel{A^t}{\longrightarrow}&E\\
\downarrow&&\downarrow\\
K&\stackrel{\phi^t}{\longrightarrow}&K
\end{array}$$

If $\Lambda\subset K$ is a $\phi$-invariant subset,  then $\pi^{-1}(\Lambda)\subset E$ is $\cA$-invariant, and  we  call
\emph{the restriction of $\cA$ to $\Lambda$}  the restriction of $\{A^t\}$ to $\pi^{-1}(\Lambda)$.

 \subsection{Hyperbolicity and dominated splitting on linear cocycles}

 \begin{defi}
 Let $\phi$ be a topological flow on a compact metric space $M$ and a $\phi$-invariant connected compact subset $\La$. We consider a vector bundle  $\pi\colon E\to \La$ and
   a linear cocycle $\cA$ over $(\La,X)$.

 We say that $\cA$ admits a \emph{dominated splitting over $\Lambda$} if
 \begin{itemize}\item There exists a splitting $E=E^1\oplus\dots\oplus E^k$ over $\Lambda$ into $k$ sub-bundles
 \item  The dimension of the sub-bundles is constant, i.e. $dim(E^i_x)=dim(E^i_y)$ for all $x,y\in\Lambda$ and $i\in \set{1\dots k}$,
     \item The splitting is invariant, i.e. $A^t(x)(E^i_x)=E^i_{\phi^t(x)}$ for all $i\in \{1\dots k\}$,

\item There exists a $t>0$ such that for every $x\in \Lambda$ and any pair of non vanishing vectors $u\in E^i_x$ and $v\in E^j_x$, $i<j$ one has
 \begin{equation}\label{e.dom}
 \frac{\norm{A^t(u)}}{\norm u}\leq \frac{1}{2}\frac{\norm{ A^t(v)}}{\norm v}
 \end{equation}

 We denote $E^1\oplus_{_\prec}\dots\oplus_{_\prec} E^k$, or $E^1\oplus_{{_\prec}_t}\dots\oplus_{{_\prec}_t} E^k$ if one wants to emphasis the role of $t$: in that case one says that
 the splitting is \emph{$t$-dominated}.
\end{itemize}
 \end{defi}

 A classical result (see for instance \cite[Appendix B]{BDV})  asserts that the bundles of a dominated splitting are always continuous. A given cocycle may admit several dominated
  splittings.  However, the dominated splitting is unique if one prescribes the dimensions $dim(E^i)$.

 One says that one of the bundle $E^i$ is \emph{(uniformly) contracting} (resp. \emph{expanding}) if there is $t>0$ so that for every
 $x\in\La$ and every non vanishing vector   $u\in E^i_x$ one has $\frac{\norm{A^t(u)}}{\norm u}<\frac 12$
 (resp. $\frac{\norm{A^t(u)}}{\norm u}<\frac 12$). In both cases one says that $E^i$ is \emph{hyperbolic}.

 Notice that if $E^j$ is contracting (resp. expanding) then the same holds for any $E^i$, $i<j$  (reps. $j<i$).

\begin{defi}
 We say that the linear cocycle  $\cA$ is \emph{hyperbolic over $\Lambda$} if
  there is a dominated splitting $E=E^s\oplus E^u$ over $\Lambda$ into $2$  hyperbolic sub-bundles so that $E^s$ is uniformly contracting and $E^u$ is
  uniformly expanding.

  One says that $E^s$ is the \emph{stable bundle}, and $E^u$ is the \emph{unstable bundle}.
 \end{defi}

 The existence of a dominated splitting or of a hyperbolic splitting is a robust property

 \subsection{Linear Poincar\'e flow}

 Let $X$ be a $C^1$ vector field on a compact manifold $M$.  We denote by $\phi^t$ the flow of $X$.

 \begin{defi} The \emph{normal bundle} of $X$ is the vector bundle $N_X $ over $M\setminus Sing(X)$ defined as follows: the fiber $N_X(x)$ of $x\in M\setminus Sing(X)$ is
 the quotient space of $T_xM$ by the  line $\RR.X(x)$.
 \end{defi}
 Note that, if $M$ is endowed with a Riemannian metric, then $N_X(x)$ is canonically identified with the orthogonal space of $X(x)$:
 $$N_X=\{(x,v)\in TM, v\perp X(x)\} $$

Consider $x\in M\setminus Sing(M)$ and $t\in \RR$.  Thus $D\phi^t(x):T_xM\to T_{\phi^t(x)}M$ is a linear automorphism mapping $X(x)$ onto $X(\phi^t(x))$. Therefore
$D\phi^t(x)$ passes to the quotient as an linear map $\psi^t(x)\colon N_X(x)\to N_X(\phi^t(x))$:

$$\begin{array}[c]{ccc}
T_xM&\stackrel{D\phi^t}{\longrightarrow}&T_{\phi^t(x)}M\\
\downarrow&&\downarrow\\
N_X(x)&\stackrel{\psi^t}{\longrightarrow}&N_X(\phi^t(x))
\end{array}$$
where the vertical arrows are the canonical projections of the tangent space to the normal space. We call this map as the \emph{linear Poincar\'e flow}.

\begin{defi}
 We say that a vector field $X$ is \emph{hyperbolic over $\Lambda$} if
  there is a dominated splitting $TM=E^s\oplus\RR .X(x)\oplus E^u$ over $\Lambda$ into $2$  hyperbolic sub-bundles so that $E^s$ is uniformly contracting and $E^u$ is
  uniformly expanding.

  One says that $E^s$ is the \emph{stable bundle}, and $E^u$ is the \emph{unstable bundle}.
 \end{defi}

 Note that if $X$ is non singular the linear Poincar\'e flow is a linear cocycle.

 Notice that the notion of dominated splitting for non-singular flows is sometimes better expressed in terms of the linear Poincar\'e flow: for instance,
 the linear Poincar\'e flow of a robustly transitive non singular vector field
 always admits a dominated splitting, when the flow by itself may not admit any dominated splitting.
 An example of a diffeomorphism with a robustly transitive set having dominated splitting into two bundles, that none of them is contracting or expanding is exhibited in \cite{BV}. The suspension of this diffeomorphism would not have a dominated splitting of the tangent space.

 \subsection{Extended linear Poincar\'e flow}

 We are dealing with singular flows and the linear Poincar\'e flow is not defined on the singularities of the vector field $X$.
 However we can extend the linear Poincar\'e flow  as defined in \cite{GLW}.

 This flow will be a linear co-cycle define on some linear bundle over $M$ (even over the singularities of $X$), that we define now.

 \begin{defi}Let $M$ be a  manifold of dimension $d$.
 \begin{itemize}
 \item We call \emph{the projective tangent bundle of $M$}, and denote by $\Pi_\PP\colon \PP M\to M$, the fiber bundle whose fiber $\PP_x$ is
 the projective space of the tangent space $T_xM$: in other word, a point $L_x\in \PP_x$ is a $1$-dimensional vector subspace of $T_xM$. There is a natural projection  $\pi:TM\to \PP M$ than takes a $1$-subspace of $TM$ to the corresponding element of $\PP M$  
  \item We call \emph{normal bundle of $\PP M $} and we denote by $\Pi_\cN\colon \cN M\to \PP M$, the $d-1$-dimensional vector bundle over $\PP M$ whose fiber $\cN_{L}$ over
  $L\in \PP_xM$
  is the quotient space $T_x M/L$.

  If we endow $M$ with riemannian metric, then $\cN_L$  is identified with the orthogonal hyperplane of $L$ in $T_xM$.
 \end{itemize}
 \end{defi}

 Let $X$ be a $C^r$ vector field on a compact manifold  $M$, and $\phi^t$ its flow. The natural actions of the derivative of $\phi^t$ on $\PP M$ and $\cN M$ define
 $C^{r-1}$ flows on these manifolds.  More precisely, for any $t\in \RR$,

 \begin{itemize}
  \item We denote by $\phi_{\PP}^t\colon\PP M\to \PP M$ the flow defined by $$\phi_{\PP}^t(L_x)= D\phi^t(L_x)\in \PP_{\phi^t(x)}.$$

  \item We denote by $\psi_{\cN}^t\colon\cN \to \cN $ the $C^{r-1}$ flow  whose restriction to a fiber $\cN_L$, $L\in \PP_x$,
  is the linear map onto
  $\cN_{\phi^t_{\PP}(L)}$ defined as follows: $D\phi^t(x)$ is a linear map from $T_xM$ to $T_{\phi^t(x)}M$, which maps the line $\pi^{-1}(L)\subset T_xM$
  onto the line $\pi^{-1}(\phi^t_{\PP}(L))$.  Therefore it pass to the quotient in the announced linear map.
  $$\begin{array}[c]{ccc}
T_xM &\stackrel{D\phi^t}{\longrightarrow}&T_{\phi^t(x)}M\\
\downarrow&&\downarrow\\
\cN_L&\stackrel{\psi^t_{\cN}}{\longrightarrow}&\cN_{\phi^t_{\PP}(L)}
\end{array}$$

 \end{itemize}


 The one-parameter family   $\psi^t_\cN$ defines a flow on $\cN $,  which is a linear co-cycle over $\phi^t_\PP$.
 We call  $\psi^t_\cN$ the \emph{extended linear Poncar\'e flow}.
 We can summarize  by the following diagrams: 


$$\begin{array}[c]{ccc}
\cN &\stackrel{\psi^t_{\cN}}{\longrightarrow}&\cN \\
\downarrow&&\downarrow\\
\PP M &\stackrel{\phi_{\PP}^t}{\longrightarrow}&\PP M\\
\downarrow&&\downarrow\\
M&\stackrel{\phi^t}{\longrightarrow}&M
\end{array}$$


\begin{rema} The extended linear Poincar\'e flow is really an extension of the linear Poincar\'e flow defined in the previous section; more precisely:

 Let $S_X\colon M\setminus Sing(X)\to \PP M$ be the section of the projective bundle defined as $S_X(x)$ is the line $\langle X(x)\rangle\in \PP_x$ generated by $X(x)$.
 Then $N_X(x)= \cN_{S_X(x)}$ and the linear automorphisms $\psi^t\colon N_X(x)\to N_X(\phi^t(x))$ and $\psi^t_\cN\colon \cN_{S_X(x)}\to \cN_{S_X(\phi^t(x))}$ are the same.
\end{rema}

\subsection{Strong stable, strong unstable and center spaces associated to a hyperbolic singularity.}
Let $X$ be a vector field and $\sigma\in Sing (X)$ be a hyperbolic singular point of $X$.
Let
 $\lambda^s_k\dots\lambda^s_2<\lambda^s_1<0<\lambda^u_1<\lambda^u_2\dots \lambda^u_l$ be the Lyapunov exponents of $\phi_t$ at $\sigma$ and let
 $E^ s_k\oplus_{_<}\cdots E^s_2\oplus_{_<}E^s_1\oplus_{_<}E^u_1\oplus_{_<}E^u_2\oplus_{_<}\cdots \oplus_{_<}E^s_l$ be the corresponding (finest) dominated
 splitting over $\sigma$.

 A subspace $F$ of $T_\sigma M$ is called a \emph{center subspace} if it is of one of the possible form below:
 \begin{itemize}
  \item Either $F=E^ s_i\oplus_{_<}\cdots E^s_2\oplus_{_<}E^s_1$
  \item Or $F=E^u_1\oplus_{_<}E^u_2\oplus_{_<}\cdots \oplus_{_<}E^u_j$
  \item Or else $F=E^ s_i\oplus_{_<}\oplus_{_<}E^s_1\oplus_{_<}E^u_1\oplus_{_<}\cdots \oplus_{_<}E^s_j$
 \end{itemize}

 A subspace  of $T_\sigma M$ is called a \emph{strong stable space}, and we denote it  $E^{ss}_{i}(\sigma)$,  is there in $i\in\{1,\dots, k\}$ such that:
 $$E^{ss}_{i}(\sigma)=E^ s_k\oplus_{_<}\cdots E^s_{j+1}\oplus_{_<}E^s_i$$

 A classical result from hyperbolic dynamics asserts that for any $i$ there is a unique infectively immersed manifold $W^{ss}_i(\sigma)$, called a
 \emph{strong stable manifold}
 tangent at $E^{ss}_i(\sigma)$ and invariant by the flow of $X$.

 We define analogously the \emph{strong unstable spaces} $E^{uu}_j(\sigma)$ and the \emph{strong unstable manifolds} $W^{uu}_j(\sigma)$ for $j=1,\dots ,l$.

We can also define the strong stable and unstable manifolds in an analogue way, for regular points $x$ in an invariant set $\La$.

\section{Multisingular hyperbolicity}
\subsection{The reparametrized linear Poincar\'e flow}\label{ss.reparametrized}

We endow the manifold $M$ with a smooth Riemannian metric $\norm{.}$.
We call \emph{reparametrizing map} to the map $h\colon \PP M\times\RR\to \RR$ defined as follows:  $h(L,t)=\frac{\|D\phi^t(u)\|}{\|u\|}$, where $u$ is a non vanishing vector in $L$.

Note that $h$ satisfies the following cocycle relation:
\begin{equation}\label{e.cocycle}
 h(L,t+s)=h(\phi^t_{\PP}(L),s)\cdot h(L,t).
\end{equation}

\begin{defi}
We call \emph{reparametrized linear Poincar\'e flow} and we denote $\Psi_t$,
the linear cocycle $\Psi^t(L,u)\colon\cN \to\cN$ as follows:
$$\Psi^t(L,u)= h(L,t)\cdot \psi_{_\cN}^t(L,u)$$
where $u\in \cN_L$.
\end{defi}

The fact that this formula defines a Linear cocycle, follows directly from the fact that $\psi_{\cN}$ is a linear cocycle and $h$ satisfies (\ref{e.cocycle}).

\subsection{Maximal invariant set and lifted maximal invariant set}

Let $X$ be a vector field on a manifold $M$ and $U\subset M$ be a compact subset.  The \emph{maximal invariant set} $\La=\La_U$ of $X$ in $U$ is the intersection
$$\La_U=\bigcap_{y\in RR} \phi^t(U).$$

 We say that a compact $X$-invariant set $K$ is \emph{locally maximal} if there exist a compact neighbourhood $U$ of $K$ so that $K=\La_U$.

\begin{defi}We call \emph{lifted maximal invariant set in $U$}, and we denote by $\La_{\PP,U}\subset \PP M$
(or simply $\La_\PP$ if one may omit the dependence in $U$), to :

 $$\La_{\PP,U}=\overline{S_X(\La_U\setminus Sing X)}\subset \PP M,$$
 where $S_X\colon M\setminus Sing X\to \PP M$ is 
$S_X(x)=\RR.X(x)$.
 \end{defi}

 The lifted maximal invariant set does not  depend upper semi-continuously on the flow when there are singularities in $U$.

\subsection{The lifted maximal invariant set and the singular points}

 The aim of this section is to find a bigger set than
 the lifted maximal invariant set $\La_{\PP,U}$ that  varies upper semi-continuously with the vector field. We do this by adding some subset of the projective space over the singular
 points, as in \cite{BdL}. All the proofs of the following lemmas and propositions can be found there.

 Let $U$ be a compact region, and $X$ be a vector field, and  $\sigma$ 
  be a hyperbolic singularity of $X$, contained in the interior of $U$.

We define the \emph{escaping stable space of $\sigma$ in $U$} $E^{ss}_{\sigma,U}$ as the biggest strong stable space $E^{ss}_j(\sigma)$ such that the invariant manifold tangent to it
(that we call \emph{escaping strong stable manifold}) $W^{ss}_j(\sigma)$  is such that all orbits in it, escape $U$. That is,

$$W^{ss}_j(\sigma) \text{ is such that } \Lambda_{X,U}\cap W^{ss}_j(\sigma)=\{\sigma\}.$$

We define the \emph{escaping  unstable space of $\sigma$ in $U$}  and the \emph{escaping strong  unstable manifolds} analogously.

We define the \emph{central space of $\sigma$ in $U$} and we denote $E^c_{\sigma,U}$ the space such that
$$T_\sigma M=E^{ss}_{\sigma,U}\oplus E^ c_{\sigma,U}\oplus E^ {uu}_{\sigma,U}$$

 We denote by
 $\PP^i_{\sigma,U}$ the projective space of $E^ i(\sigma,U)$ where $i=\set{ss,uu,c}$.

 \begin{lemm}\label{l.escaping} Let $U$ be a compact region and $X$  a vector field whose singular points are hyperbolic and contained in the interior of $U$.
 Then, for any $\sigma\in Sing (X)\cap U$, one has :
 $$\La_{\PP,U}\cap\PP^{ss}_{\sigma,U}=\La_{\PP,U}\cap\PP^{uu}_{\sigma,U}=\emptyset.$$

 \end{lemm}

As a  consequence we get the following characterization of the central space of $\sigma$ in $U$:
\begin{lemm}
The central space  $E^c_{\sigma,U}$   is the smallest center space containing $\La_{\PP,U}\cap \PP_\sigma$.
\end{lemm}

We are now able to define the subset of $\PP M$ which extends the lifted maximal invariant set and which has the upper-semi continuity property.

 \begin{defi}Let $U$ be a compact region and $X$  a vector field whose singular points are hyperbolic, and disjoint from the boundary $\partial U$.
 Then the set
 $$B(X,U)=\La_{\PP,U}\cup\bigcup_{\sigma\in Sing(X)\cap U} \PP^c_{\sigma,U} \subset \PP M$$
 is called the \emph{extended maximal invariant set of $X$ in $U$}
 \end{defi}

 \begin{prop}\label{p.extended} Let $U$ be a compact region and $X$  a vector field whose singular points are hyperbolic, and disjoint from the boundary $\partial U$.

 Then the extended maximal invariant set $B(X,U)$ of $X$ in $U$ is a compact subset  of $\subset \PP M$.

 Furthermore, there is a $C^1$-neighbourhood $\cU$ of $X$
 for  which the map $Y\to B(Y,U)$ depends upper semi-continuously on $Y\in\cU$.
 \end{prop}

 A chain recurrence class admits a basis of  filtrating neighbourhood. That is, for any chain recurrence class  we can find a sequence of neighbourhoods ordered by inclusion  $U_{n+1}\subset U_n$, such that $C=\bigcap U_n$
  We define

$$\widetilde{\La(C)}=\bigcap_n\widetilde{\La(X,U_n)}  \text{ and } B(C)=\bigcap_n B(X,U_n).$$
These two sets are independent of the choice of the sequence $U_n$.

\subsection{Strong multisingular hyperbolicity}

 We are now ready to define the notion of hyperbolicity we will use in this paper. It is expressed in term of the reparametrized linear Poincar\'e flow defined in
 Section~\ref{ss.reparametrized}:

 \begin{defi}Let $U\subset M$ be a compact region and $X$ a $C^1$-vector field on $M$ and $C$ a chain recurrence class of $X$.  We say that $X$ is \emph{strong multisingular hyperbolic in $C$}
 if $X$ has hyperbolic  singularities in $U$ and if the restriction of the reparametrized linear Poincar\'e flow $\frac{\|D\phi^t(u)\|}{\|u\|}\psi^t$ to
 $B(C)$ is a uniformly hyperbolic linear cocycle over $\phi_\PP$.

  For $L\in B(C)$ we denote,
  $$\cN_L= \cN^s(L)\oplus \cN^u(L)$$
  the stable and unstable spaces of the reparametrized linear Poincar\'e flow.

 We call $dim\cN^s(L)$ the  $s$-index of multisingular hyperbolicity of $X$.
 \end{defi}

Note that the reparametrized linear Poincar\'e flow $\Psi^t$ is one of the possible reparametrizing cocycles considered in \cite{BdL} and therefore if  $\Psi^t$  is hyperbolic then $X$ is multisingular hyperbolic according to the definition in \cite{BdL}.
For simplicity we work with this stronger definition since our aim is just to make an example.
Being multisingular hyperbolic is a robust property and the proof of that is also in \cite{BdL}.

 \begin{coro}[in \cite{BdL}]\label{c.robusthyp}If $X$ is multisingular hyperbolic in $U$ then $X$ is a star flow in $U$.
 \end{coro}


\subsection{Extension of hyperbolicity along an orbit}

Let us consider now a linear cocycle $\cA$ of a manifold $M$, a hyperbolic set $\La$ for the cocycle $\cA$ and an orbit $y$ such that the $\alpha $-limit of $y$, and the $\omega$-limit of $y$ are in $\La$. 


 Then the next lemma shows conditions in which  we can extend the hyperbolic structure of our cocycle to $\La\cup o(y)$.

\begin{lemm}\label{l.masuna} Let $\La$ be a  hyperbolic, maximal invariant set in $U$, for  $\cA$, and  $E_{\La}=E^s\oplus E^u $ be its hyperbolic splitting. Suppose as well that there is a wondering point $y$, sich that its orbit $O(y)$ satisfie the following:
 \begin{itemize}\item The $\alpha $-limit of $y$, $\alpha(y)$ is in $\La$, and therefore there is a stable conefield that is invariant for the future along the orbit of $y$

\item The $\omega $-limit of $y$ $\omega (y)$ is in $\La$ and therefore there is an unstable conefield that is invariant for the past  along the orbit of $y$
\item These conefields intersect transversaly
  \item There exists a compact neighbourhood $U'$ such that $\La\cup O(y)$ is a maximal invariant set in $U'$.

\end{itemize}
Then there exist  a unique hyperbolic splitting along $O(y)$   such that
the set $\La\cup O(y)$  is hyperbolic with that splitting.

\end{lemm}
\begin{proof}
Let us consider the unstable space of $\omega (y) \subset \La$. Since  $\La$ is   hyperbolic then there is a splitting  $$E_{\omega (y)}=E^s_{\omega (y)}\oplus E^u_{\omega (y)} $$
that  extends by continuity to a small neighborhood $u_{\omega}$ of the $\omega $-limit of $y$,
and so there  exist $T>0$ such that the splitting extends to  $\phi^t(y)$  for  $t>T$. 

There is  an unstable cone field $C^u(\phi^t(y))$ around the unstable space of  $\phi^t(y)$  for  $t>T$ that is invariant for the past and therefore  extends along the orbit  for the past. We define
$E^s_y$ as the set of vectors which do not enter in this unstable cone field for any large positive iterate.

 The dimension of $dim(E^s_y)$ must the be $$dim(M)-dim(E^u_{\omega (y)})-1\,.$$  Then $dim(E^s_y)=dim(E^s_{\omega (y)})$.
We define $E^u_y$ analogously.
By construction the dimensions must match. The transversality of the cone fields gives us that this spaces form a hyperbolic splitting along the orbit. 

With this splitting $\La\cup O(y)$  is hyperbolic. The continuity comes form the fact that the unstable and stable cone fields along the orbit of $y$ coincide with the cone fields given by the hyperbolicity of $\La$ around the piece of orbit of $y$ that never leaves $u_{\omega}$ for the future.

\end{proof}


\begin{prop}\label{p.unamasextended}
Suppose that $\La$ is a multisingular hyperbolic, maximal invariant set in $U$. Suppose as well that
\begin{itemize}
\item $y$ is such that the $\alpha$ and $\omega$ limits of $y$, $\alpha(y)$  and $\omega(y)$ are in $\La$. 
  \item There exists a compact neighborhood $U'$ such that $\La\cup O(y)$ is a maximal invariant set in $U'$,
  \item The orbit of $y$ does not intersect any escaping stable or unstable manifold of any singularity in $\La$
  \end{itemize}
Then the extended maximal invariant set $\La_{\PP}(X,U')$ is $\La_{\PP}(X,U)\cup O(L)$ where $L=S_X(y)$ and $O(L)$ is the orbit of $L$ by $\phi^t_{\PP}$.

\end{prop}
\proof
 The set $S_X(\La_{U'} \setminus Sing(X))$ is one to one with respect to $\La_{U'} \setminus Sing(X)$ . Therefore $$S_X(\La_{U'} \setminus Sing(X))=S_X(\La_{U} \setminus Sing(X))\cup O(L)\,.$$
The hypothesis above, that state that the orbit of $y$ is away from the escaping stable and unstable manifolds of the singularity, and the fact that  the $\alpha$ and $\omega$ limits of $y$ are in $\La$
$$S_X(\La_{U} \setminus Sing(X))\cup \overline{O(L)}\subset S_X(\La_{U} \setminus Sing(X))\cup O(L)\,.$$
Therefore \begin{eqnarray*}
\overline{S_X(\La_{U'} \setminus Sing(X))}&=&\overline{S_X(\La_{U} \setminus Sing(X))}\cup \overline{O(L)}\\&=&\overline{S_X(\La_{U} \setminus Sing(X))}\cup O(L)\\&=&\La_{\PP}(X,U)\cup O(L)\,.\end{eqnarray*}
\endproof

\begin{coro}\label{c.unamashyp}

Suppose that $\La$ is a  multisingular hyperbolic, maximal invariant set in $U$, for  $X$. Suppose as well that there is a point $y$ such that:
\begin{itemize}\item  The $\alpha$ and $\omega$ limits of $y$, $\alpha(y)$  and $\omega(y)$ are in $\La$. 
  \item There exists a compact neighborhood $U'$ such that $\La\cup O(y)$ is a maximal invariant set in $U'$,
  \item The orbit of $y$ does not intersect any escaping stable or unstable manifold of any singularity in $\La$
   \item The stable and unstable conefields  along the orbit of $S_X(y)$ (that arise from the hyperbolic splittings of $\alpha(y)$  and $\omega(y)$ respectively) intersect transversally, 
 \end{itemize}
Then  $\La\cup O(y)$ is  multisingular hyperbolic.
\end{coro}
Let us consider the set of chain recurrent  points in a maximal invariant set $\La\cap\cR$ and suppose that this set is maximal invariant in a smaller neighborhood $U'$, i.e. $$\bigcap\phi^t(U')=\La\cap\cR\,.$$

Applying the same argument to a set of orbits in the hypothesis of Proposition \ref{p.unamasextended}, we get that if the non chain recurrent orbits in a maximal invariant set  do not intersect the escaping spaces of the singularities, then $$B(X,U')\cup S(\La\cap\cR ^c)=B(X,U)\,.$$
As a consequence:
\begin{coro}\label{c.muchasmashyp}
Let $\La$ be the maximal invariant set in $U$. We consider the set of the chain recurrent orbits $\La\cap\cR$ and the set of the non chain recurrent orbits $\La\cap\cR^c$.
We lift the chain recurrent orbits $S(\La\cap\cR)$. Suppose as well that :
\begin{itemize}\item The set of chain recurrent  orbits in the extended maximal invariant set $B(X,U)\cap S(\La\cap\cR)$ is hyperbolic for the reparametrized linear Poincar\'{e} flow with the same index for all connected components.
  \item Every non chain recurrent orbit $O(y) \in \La$ does not intersect any $U$ escaping stable or unstable manifold of any singularity in  $\La$
   \item The stable and unstable conefields that extend along  the lifted non chain recurrent orbits   intersect transversally,  \end{itemize}
Then  $\La$ is multisingular hyperbolic.
\end{coro}

\section{A strong multisingular hyperbolic set in $S^3$}\label{attractor-repeller}
This section will be dedicated to  building a set in $S^3$ containing 2 singularities of different indexes that will be strong multisingular hyperbolic.  However this set will not be chain recurrent.

\begin{defi}\label{d.sll} We say that a hyperbolic singularity is strong Lorenz like if its tangent space splits into 3 invariant spaces. If the stable index is 2 then the  Lyapunov exponents satisfy :
 $$0<-\lambda_a^{s}<\lambda_a^{u}<-\lambda_a^{ss}\,.$$
 If the unstable index is 2 then: $$-\lambda_r^{uu}<\lambda_r^{s}<-\lambda_r^{u}<0$$
 \end{defi}

Recall that this section is dedicated to prove:
\begin{theo*}
There exists an open set of vector fields $\mathcal{U}\subset\mathcal{X}^1(S^3)$ such that every $X\in\mathcal{U}$ has  the following properties.
\begin{itemize}
\item There is a filtrating region $U=U_a\cap U_r$
where $\La$ is the maximal invariant in it i.e. $$\La=\bigcap_{t\in \RR}\phi^t(U)$$ where $\phi$ is the flow of $X$.
\item  All singularities contained in $\La$ are strong Lorenz like.
\item The set $\La$ contains a  singularity $\sigma_a$ that is accumulated by periodic orbits and that has an escaping strong stable manifold with respect to $U_a$.
\item $\La$ contains a  singularity $\sigma_r$ that is accumulated by periodic orbits and that has an  escaping strong unstable manifold with respect to $U_r$.

\item There are non-chain recurrent orbits  in $\La$ such that the $\alpha$-limit of them is in the chain-recurrent class of $\sigma_r$ (that we call $L_r$ ) and the $\omega$-limit of them is in the chain-recurrent class of $\sigma_a$ (that we call $L_a$).
\item The set $\La$ is strong multisingular hyperbolic.

\end{itemize} 

\end{theo*}

\subsection{The Lorenz attractor and the stable foliation}
In this subsection we will shortly comment on the construction of a geometric Lorenz attractor, done in \cite{GuWi}.
\subsubsection{Guckenheimer Williams, geometric model}
We consider a flow in $\RR^3$ as in \cite{GuWi}, having a robustly transitive  singular attractor, that we call $L_a$.
This set has the following properties:
\begin{itemize}
\item It has a singularity $\sigma_a$ in the origin with three different  real  Lyapunov exponents $\lambda_1,\lambda_2,\lambda_3$, with the following relation:$$-\lambda_2>\lambda_1>-\lambda_3>0\,,$$
The expansion rate $\lambda_1$ is bounded form below by $\sqrt{2}$ and from above by $2$ 
\item For the attractor $L_a$, we can consider an attracting region $W_a$  such that the boundary of this neighbourhood is a  bi-torus.

\item The strong stable spaces  of the points in $W_a$ are well define 
\end{itemize}
Additionally we ask that the strong contraction rate is bigger that $4$ and smaller than $5$.

\subsubsection{Attracting region}\label{ball}

Since we aim to construct an example in $S^3$,  it will be more convenient  to work with an attracting region $U_a$ witch is a ball.
The original construction of Guckenheimer Williams defines a vector field in $\RR^3$ such that
there exist an attracting region   $U_a$ such that:
\begin{itemize}
\item $W_a\subset U_a$
\item the maximal invariant set contained in it is $L^a $ and  two saddle singularities of stable index one called $p$ and $p'$
\item the strong stable manifold of all points in   $U_a$  is well defined and parallel to the stable manifold of $\sigma_a$.
\item We can choose  $U_a$ such that its boundary  is diffeomorphic to $S^2$.
\end{itemize}
 For a more detail description we refer  the reader to  Guckenheimer Williams's work \cite{GuWi} .

Let us now consider an arc $\gamma$ in a branch of the stable manifold of $p$ (not containing $p$) and a cylinder $C_p$  with axis $\gamma$. This cylinder can be parametrized by  $S^1\times[\delta,\rho]$ and if the radius is small enough, the stable manifolds of the points in $U_a$ cut  $C_p$  in a foliation parallel to the axis. We can find a compact neighborhood $C_a\subset U_a$ of $\gamma$ such that its boundary is a smooth cylinder that contains $C_p$.

 In addition to this we ask that this cylinder  $C_a$ cuts the boundary of $U_a$.

Now we consider $U_a\setminus C_a$. 
We can change coordinates so that  $C_p$ is an annulus $A_a$ and now the parallel foliation induced by the stable manifolds of the points of $U_a$ in $C_p$ is radial. 
such that one of the connected components of $$\partial[U_a\setminus C_a]\setminus A_a$$ has a point of intersection of the stable manifold of $p$ and doesn't have points of intersection of the stable manifolds of the other singularities. We call  this component   $D_a$.

For simplicity we continue to call  $U_a\setminus C_a$  as  $U_a$.
\subsection{A plug }
In subsection \ref{plug} we will prove the following theorem:

\begin{theo}\label{t.tube}
There exist a vector field $\chi$ such that its flow  $\phi_{\chi}$  defined  in $S^3$
 has  the following properties:

 There is a region $S^2\times{[-1,1]}\subset S^3$ such that
\begin{itemize}
\item The  vector field $\chi$ is entering at $S^2\times\set{-1}$ and points out at $S^2\times\set{1}$
\item The  vector field $\chi$ is such that the chain recurrent set consists of 2  source  singularities, $p_1$ and  $p_1'$,  2 sinks singularities, $p_2$ and $p_2'$, and  2 periodic saddles, $p_3$ and $p_3'$.
 \item The intersection of the invariant manifolds of the saddles,  with the boundary of $S^2\times{[-1,1]}$, are disjoint circles that we name as follows:
  \begin{itemize}
  \item $W^{s}(p_3)\cap S^2\times{[-1,1]} = c_0$ in $S^2\times\set{-1}$,
  \item $W^{s}(p'_3)\cap S^2\times{[-1,1]} = c'_0$ in $S^2\times\set{-1}$,
\item $W^{u}(p_3)\cap S^2\times{[-1,1]} =c_1$ in $S^2\times\set{1}$,
\item  $W^{u}(p'_3)\cap S^2\times{[-1,1]} = c'_1$ in $S^2\times\set{1}$.
 \end{itemize}
\item The circle $c_0$ bounds a disc not containing $c'_0$, that we call $D_0$. The circle $c'_0$ bounds a disc containing $c_0$, that we call $D'_0$. And they both bound an open annulus called $A_0$.  Analogously we define $D_1$, $D' _1$ and $A_1$.
\item  The orbit $O(x)$ of a point $x$ in $S^2\times\set{-1}$, crosses  $S^2\times\set{1}$ if and only if $x\in A_0$ and $O(x)\cap S^2\times\set{1}\in A_1$,

\item There is a well defined crossing map $P:A_0\to A_1$. Consider the radial foliation $V_0$ in $A_0$  Then the image of a radial foliation  under $P$ intersect transversally a radial foliation  in $A_1$ and it extends to a foliation in $A_1\cup c_1 \cup c'_1$.
 \end{itemize}
The complement of $S^2\times{[-1,1]}$ in $S^3$ are 2 balls, one in the basin of attraction of a source $r$ (that has $S^2\times\set{-1}$ in the boundary ), and the other in the basin of attraction of a sink $a$.
\end{theo}

\begin{figure}[htb]
\begin{center}
\includegraphics[width=0.78\linewidth]{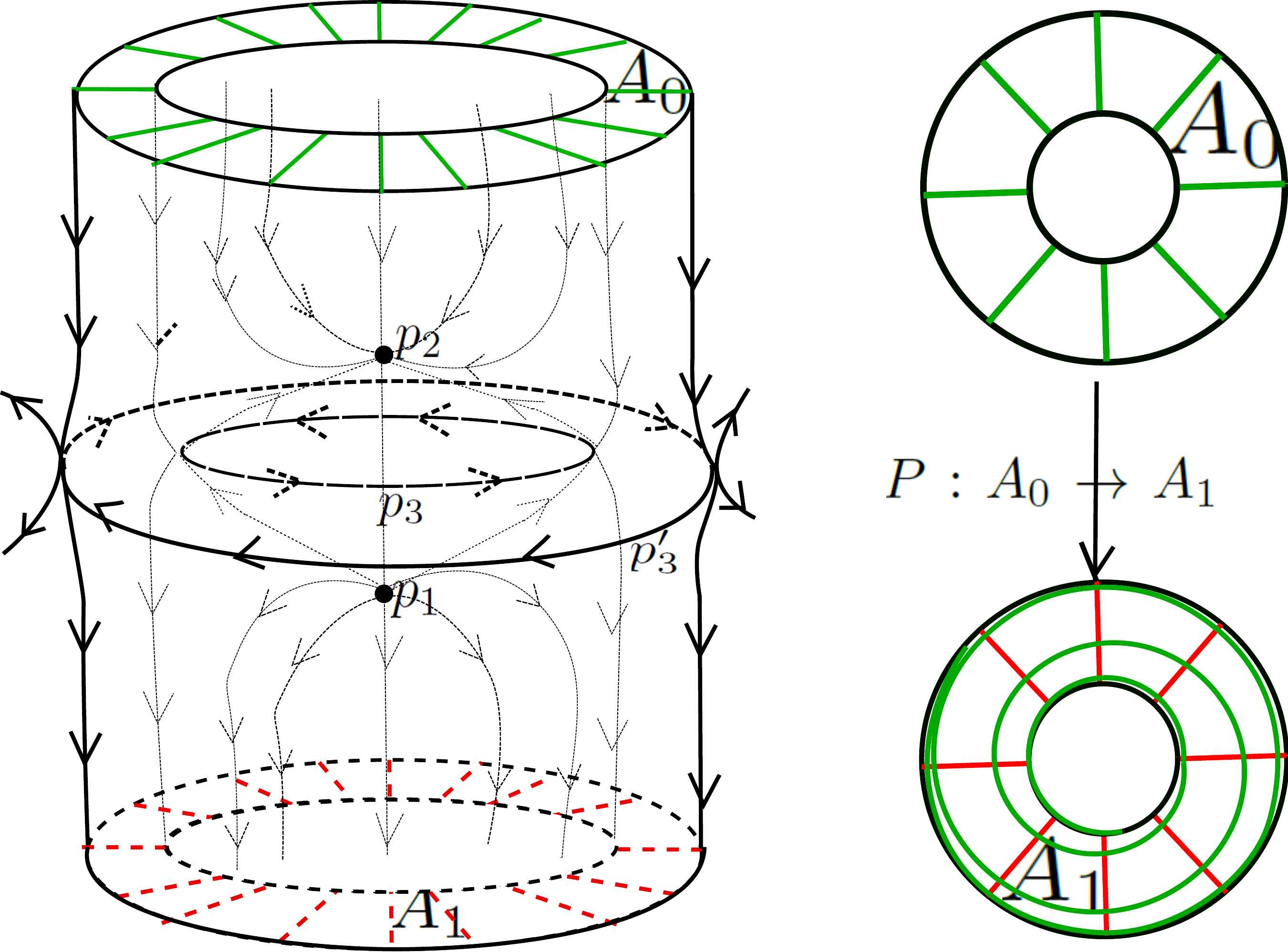}
\end{center}
\caption{The  vector field $\chi$ in  $S^2\times{[-1,1]}\subset S^3$ and the map $P:A_0\to A_1$}
\end{figure}

\begin{figure}[htb]
\begin{center}
\includegraphics[width=0.72\linewidth]{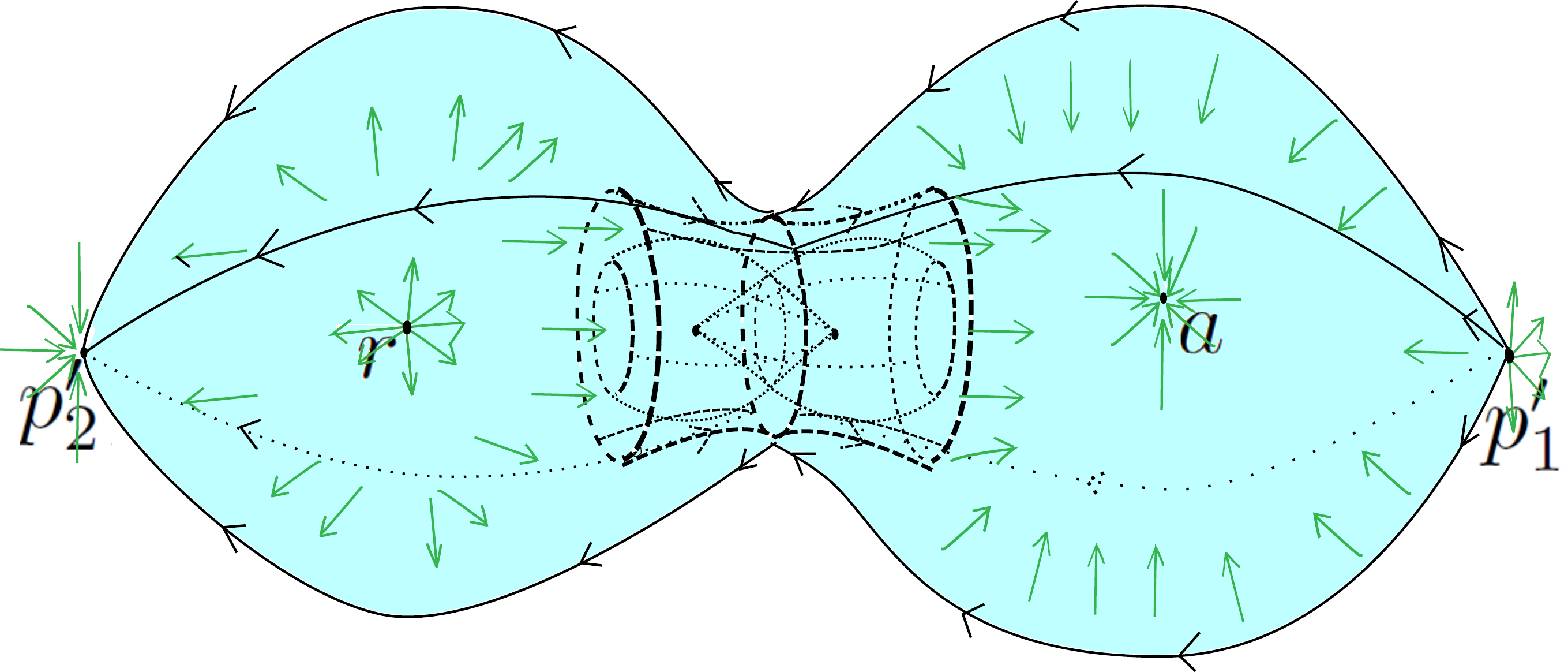}\label{compflow1.f}
\end{center}
\caption{The vector field $\chi$ in $S^3$.}
\end{figure}

This vector field will be used to select orbits in the unstable manifold of  $L_r$ and transform them in such a way that these orbits are  in the transverse intersection of the unstable manifold of  $L_r$
 with the stable manifold of $L_a$.
\subsection{Gluing the pieces: defining a flow on $S^3$}\label{completeflow3}
Let us consider the vector field $\chi$ defined above, in $S^2\times{[-1,1]}$.
Instead of completing the vector field to $S^3$ by gluing two balls that one contains a sink and the other a source, we will glue the balls $U_a$ from 
subsection {\ref{ball}  and another ball $U_r$ with a vector filed which is the reverse time of the vector field in $U_a$. we can do this since the vector field in $U_a$ is transverse to the boundary and pointing in.
\begin{figure}[htb]
\begin{center}
\includegraphics[width=0.40\linewidth]{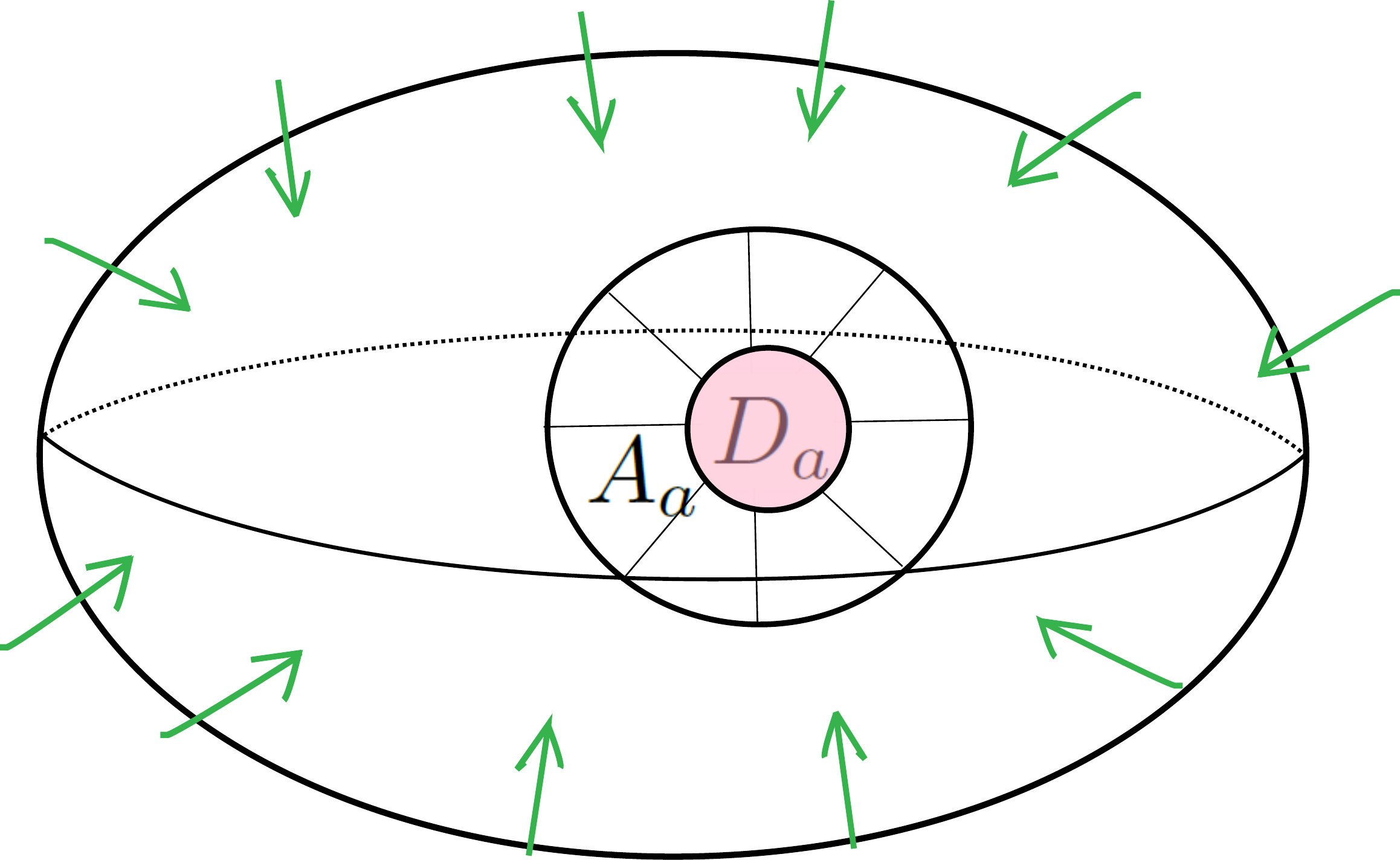}
\end{center}
\caption{The ball $U_a$.}
\end{figure}

Recall that from subsection {\ref{ball}} we have that
\begin{itemize}
  \item The boundary of $U_a$ is $S^2$
  \item There is an annulus $A_a$ in $S^2$
 such that the strong stable manifolds of $L_a$ intersect  $A_a$ along a radial foliation
  \item The annulus $A_a$ bounds a disc $D_a$ containing the intersection of the stable manifold of $p$ and does not intersect the stable manifold of  the other extra singularity.
\end{itemize}
we choose the gluing map so that 
\begin{itemize}
  \item  $A_a$ is mapped to an annulus containing $A_1$, a radial foliation of $A_a$ is send to cut $A_1$ in a radial foliation.
  \item $D_a$ is mapped to the interior of  $D_1$.
\end{itemize}
\begin{figure}[htb]
\begin{center}
\includegraphics[width=0.63\linewidth]{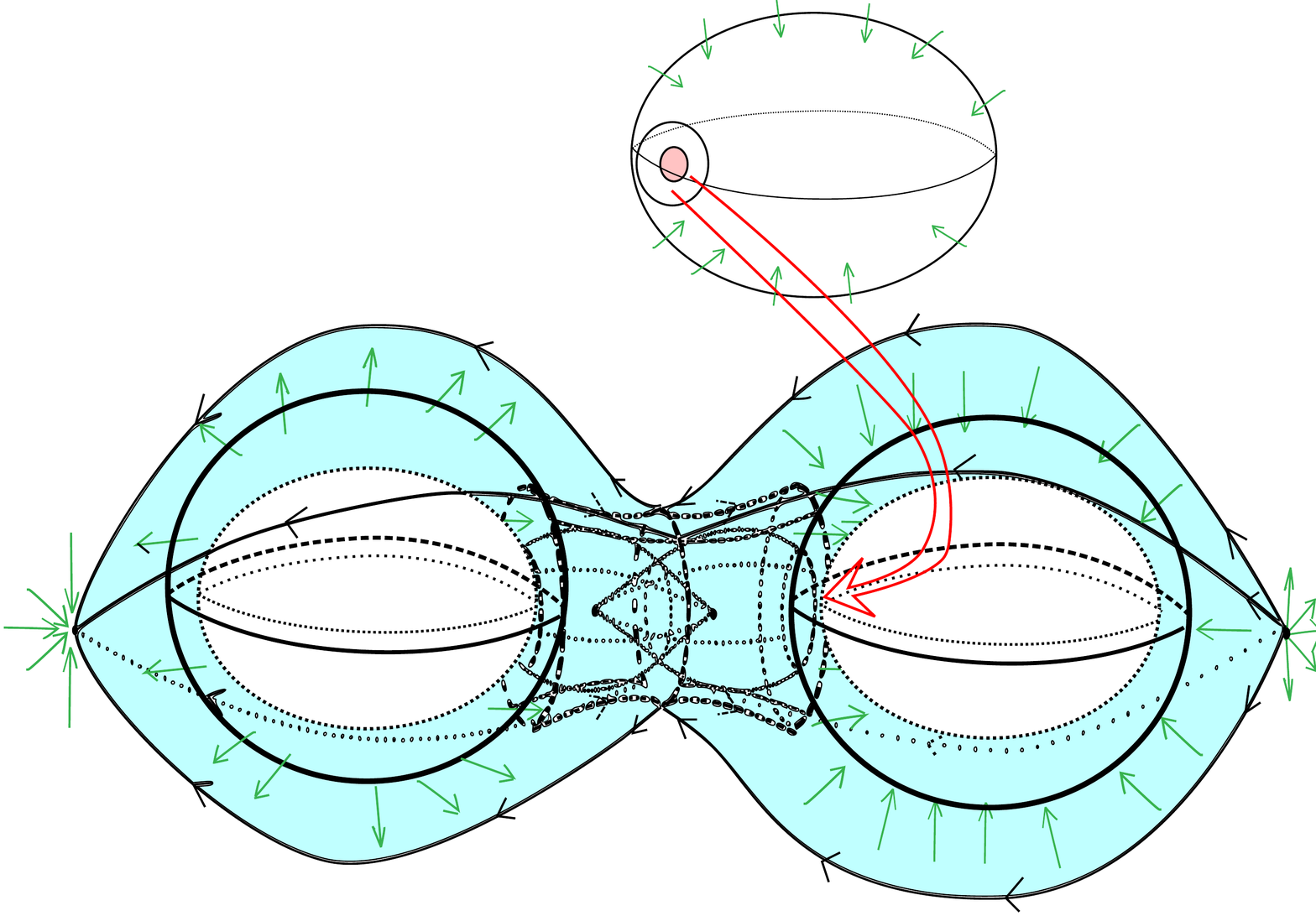}
\end{center}
\caption{Gluing $U_a$ to $S^2\times\set{1}$ }
\end{figure}
We do the same for $U_r$

Note that by doing this process we do not create any new recurrent orbits.
We call the resulting vector field in $S^3$,
$X$.

\subsection{The  filtrating neighborhood}

Let us consider $X$ from subsection \ref{completeflow3}.

If we remove some small neighborhoods inside the basin of  the 2 sources $p_1$ and  $p_1'$ , we get a repelling region $V_r$.

If we remove some small neighborhoods inside the basin of  the  2 sinks $p_2$ and $p_2'$, we get an attracting region $V_a$.

The resulting open set $$U=V_a\cap V_r$$
is a filtrating  neighborhood. We call the maximal invariant set in it $\Lambda$.

\begin{lemm}
For the vector field $X$ the maximal invariant set $\Lambda\subset U$ is  strong multi singular hyperbolic.
\end{lemm}
\proof
The Lorenz attractor is singular hyperbolic, i.e. $$T_xS^3=E^{ss}\oplus E^{cu}\,\,\,\,\,\,\,\,\text{ for all } x\in L_a$$ (see \cite{MPP}). The strong stable space of $L_a$ is escaping, and therefore the center space is  $E^{cu}$. The singularities in $L_a$ are strong Lorenz like, and in fact, the expansion rate can never be bigger that $2$ while the contraction rate is always bigger that $4$. As a consequence $$\Psi^t(L,u)= h(L,t)\cdot \psi_{_\cN}^t(L,u)$$ still contracts $\cN^s(L)$ since the biggest possible expansion rate for  $h(L,t)$ is smaller than $2$ . Since $E^{cu}$ expands volume, that means that , $$\Psi^t(L,u)= h(L,t)\cdot \psi_{_\cN}^t(L,u)$$  expands $\cN^u(L)$.

The periodic orbits are also  strong multisingular hyperbolic since $h(L,t)$ does not  expand or contract exponentially along a periodic orbit.

We need to check the the strong multisingular hyperbolicity in the wondering orbits that go from $A_0$ to $A_1$.
For this, lemma \ref{c.muchasmashyp} tells us we need to check that the stable and unstable spaces that extend along this orbits, intersect transversely.
This is a consequence of lemma \ref{l.transvers} and the fact that the stable foliation of $L_a$ intersects $A_1$ radially, and the unstable foliation of $L_r $ intersects $A_0$ radially.
\endproof

\section{A multisingular hyperbolic set in $M^5$}\label{s.main}
The aim of this section is to find a chain recurrent set that is multisingular hyperbolic with 2 singularities of different indexes.
 For this, the strategy will be to multiply the vector field $X$ in $S^3$ from section  \ref{attractor-repeller} times  $\mathbb{RP}^2$ with a simple vector field. Then modify the resulting set to obtain new recurrence.

The following Lemma will be proven in subsection\ref{ss.propY}.
\begin{lemm}\label{propY}
There exist a vector field $Y$ in $\mathbb{RP}^2$ with the following properties:
\begin{itemize}
\item $Y$ is a $\mathcal{C}^{\infty}$ vector field

\item It has 3 singularities: a saddle singularity $s$, a source $\alpha$ and a sink $\omega$. 
\item The contracting and expanding  Lyapunov exponents of the saddle are equal in absolute value ($\lambda_{sss}=-\lambda_{uuu}$), and $\lambda_{uuu}>> 6$.
\item One of the stable branches of $s$ (that is an orbit) has $\alpha$-limit  $\alpha$.
and one of the unstable branches of $s$ (that is an orbit) has $\omega$-limit   $\omega$.
\item The other two branches form an  orbit with $\alpha$-limit and $\omega$-limit in $s$ and we call this orbit $\gamma$.
\item There is a transverse section to $\gamma$ and to the flow, that we call $T=[-1,1]\times\set{a}$,  $T\cap\gamma=0\times{a}$ and the flow of $Y$, $\phi^Y(s,t)$ is such that:
    \begin{itemize}
    \item  If $s=(x,a)\in T$ is such that $x>0$, then $\phi^Y(s,t)$  does not cross $T$  for any $t>0$ and has $\omega$-limit in $\omega$. And for $t<0$ there exists only one $t_s<0$ such that $\phi^Y(s,t_s)=s'\in T$ with $s'=(x',a)$, $x'<0$ and the $\alpha$-limit of $s$ is $\alpha$.
    \item If $s=(x,a)\in T$  and $x<0$, then  $\phi^Y(s,t)$ does not cross $T$  for any $t<0$ and has $\alpha$-limit in $\alpha$. And for $t>0$ there exists only one $t_s>0$ such that $\phi^Y(s,t_s)=s'\in T$ with $s'=(x',a)$, $x'>0$ and the $\omega$-limit of $s$ is $\omega$.

    \end{itemize}
\end{itemize}
\end{lemm}

\subsection{The vector field in $M^5$}
We start by considering the vector field $Z_{id}=(X,Y)$ in the manifold $M^5=S^3\times\mathbb{RP}^2$ and it's flow $\phi_{id}$ .
Let us define the section $$\sum=S^3\times T$$ which is transverse to $Z_{id}$, and a flow-box $\sum\times[-1,0]$. 
\begin{prop}
Let $H:\sum\to\sum$ be a $\mathcal{C}^{\infty}$ diffeomorphism isotopic to identity and that is the identity on the boundary. 
There exist a  flow $\phi_{H} $ in  $M^5$, associated to a $\mathcal{C}^{1}$ vector field $Z_H$, such that $\phi_H=\phi_{id}$ in the complement of the flow-box $\sum\times[-1,0]$, and in the flow-box $(H(z),0)=\phi_H((z,-1),1)$.

\end{prop}
\proof
Since $H $ is isotopic to the identity we have that there exist a diffeomorphism $F:\sum\times[-1,0]\to\sum$ such that $F(\sum, -1)=id$ and $F(\sum, 0)=H$.
We also have that there exist $F':\sum\times[-1,0]\to\sum$ such that $F'(\sum, -1)=H^{-1}$ and $F'(\sum, 0)=id$.
Let us define the flow $\phi_H$ as follows:
\begin{itemize}
\item $\phi_H(y,t)=\phi_{id}(y,t) $ for every $t$ such that $\phi_H(y,t)\notin \sum\times[-1,0]$
\item If $t_0$ is such that $\phi_H(y,t_0)\in  \sum\times\set{-1}$ then $$\phi_H(y,t)=F(\phi_{id}(y,t_0),s)\,,$$  for every  $s=t-1-t_0$  such that $-1\leq  s \leq0$.
\item If $t_1$ is such that $\phi_H(y,t_1)\in  \sum\times\set{0}$ then $$\phi_H(y,t)=F'(\phi_{id}(y,t_1),s)\,,$$  for every  $-s=t-t_1$  such that $-1\leq  s \leq0$.
\end{itemize}
Now we define the vector field $Z_H$ by taking at any point,  the derivative (on $t$) of $\phi_H(y,t)$ and since $\phi_H(y,t)$ is sufficiently smooth, then so  is $Z_H$.
\endproof

\subsubsection{ A filtrating region for $Z_H$ }


We recall that $U$ is the filtrating region in $S^3$ for $X$, defined in section \ref{attractor-repeller}.
We define now the filtrating region in $M^5$ that is interesting to us:
We consider a repelling region $u_{\alpha}\subset \mathbb{RP}^2$ of $\alpha$ for $Y$, such that $\alpha$ is the maximal invariant set in $u_\alpha$. Similarly, consider  a trapping region $u_{\omega}\subset \mathbb{RP}^2$ 
We take  the respective repelling and trapping regions of this singularities in $M^5$. We define the repelling region $U_{\alpha}=S^3\times u_{\alpha}$ and the trapping region $U_{\omega}=S^3\times u_{\omega}$.
We define as well $$V=[U\times \mathbb{RP}^2]/\set{U_{\alpha}\cup U_{\omega}}\,.$$

Let us consider the maximal invariant set for $Z_{id}$ in $V$ that we call $\Lambda_{id}$.
\begin{prop}\label{p.max1}
The maximal invariant set $\Lambda_{id}$ in $V$ (for $Z_{id}$) intersects $\sum$. For any $H $ as above, any orbit in the maximal invariant set $\Lambda_H\in V$ (for $Z_{H}$) either crosses $\sum$ or is contained in  $S^3\times \set{s}$.
\end{prop}
\proof
Let us consider the saddle singularity in $Y$ that we called  $s$. By construction, there is a unique orbit of $Y$, formed by  a branch of the stable and unstable manifold of $s$,  that crosses $T$. Since the contraction and expansion rates in $Y$ are stronger than in $X$, then the points in $S^3\times\set{s}$ have a connection between the strong stable and unstable manifolds and the orbits in this connections cross $\sum$.

If the orbit $\gamma_y $ of a point $y=\set{(x,l)}$ never crosses $\sum$ then $$Z_{id}\mid_{\gamma_y}=Z_{H}\mid_{\gamma_y}.$$
Let us see that $\Lambda_{id}$ is contained in  $S^3\times \set{s}$ or it crosses $\sum$.

We take $u_0=\mathbb{RP}^2/u_{\alpha}\cup u_{\omega}$. Then the maximal invariant set in $u_0$ for $Y$ is the saddle $s$ and the saddle connection (the orbit that contains one unstable branch and one stable branch of $s$). All other points have their $\alpha$ and $\omega$-limits in the singularities $\alpha$ and  $\omega$ (see the properties of $Y $ in (\ref{propY})).
So if there is a point $y\in\gamma_y$ such that $y\notin S^3\times \set{s} $ and $\gamma_y\cap \sum=\emptyset$ then the orbit of $l$ by $Y$ has  $\alpha$ or $\omega$-limits in the singularities $\alpha$ and  $\omega$. This implies that $y$ has $\alpha$ and $\omega$-limits in $U_{\alpha}\cup U_{\omega}$. Therefore $\gamma_y\notin \Lambda_{id}$.
\endproof


 \subsection{Some usefull subsets of $S^3$}

Recall that there are 2 saddles singularities in $S^3$, $\sigma_a\subset L_a$ and $\sigma_r \subset L_r$.
Also from the construction of $X$, the wondering orbits with $\omega$-limit in $L_a$ had three possible kinds of $\alpha$-limit
\begin{itemize}
\item a  source called $p_1'$ or $p_1$
\item a chain recurrence class that is just a periodic orbit
\item a subset of $L_r$.
\end{itemize}
Analogously,  the wondering orbits with $\alpha$-limit in $L_r$ had three possible kinds of $\omega$-limit
\begin{itemize}
\item a sink called $p_2'$ or $p_2$
\item a chain recurrence class that is just a periodic orbit
\item a subset of $L_a$.
\end{itemize}

\begin{defi}

\begin{itemize}

\item We define $K_Y+ 1$ as the minimum of the times  for a point in $T$ to return to $T$ for $Y$. 
\item By construction of the Lorenz attractor (see \cite{GuWi}) there is a small linear neighborhood around the singularity $\sigma_a$, in which we can consider the coordinates $(x,y,z)$ to correspond to the strong unstable, weak stable and stable spaces respectively. We can choose thees coordinates so that singularity is approached by orbits of $L_a$ only in one semi space that corresponds to the points with positive $y$ value. We say then that $\sigma_a$ has an \emph{escaping separatrix} $W^{cs-}$ which is the half stable manifold that escapes from a neighborhood of $L_a$. 

  In the same way there is an escaping separatrix $W^{cu+}$ for the singularity $\sigma_r$ in $L_r$.
\item We  choose an orbit $\gamma_{cs} \in W^{cs-}$  such that $\gamma_{cs}$ is tangent to $y$ in the linearized neighborhood of $\sigma_a$ and is in the basin of attraction of a source (we choos $p_1'$ for convenience) for the past. There is a neighborhood of   $\gamma_{cs}$ (with $\sigma_a$ in the boundary)  that we call $D_a$  that is a repelling region.
In the same way we define $D_r$. By choosing these neighborhoods smaller  $D_a\cup D_r$ does not disconnect  $S^3$.
\item We define  the corresponding repelling region in $M^5$, $V_r=\mathbb{RP}^2\times D_r.$
\item We define a compact ball $B_a\subset S^3 $ that intersects $ U$  in a linearized region  of the singularity $\sigma_a$ and it's boundary  intersects $L_a$ only in one point $p$ of the local unstable manifold of $\sigma_a$.
In other words, inthe coordinates given above if
$ (x,y,z)\in B_a$ then $y\geq0$ and it is $0$ only in $p\in W_X^u(\sigma_a)$. 
\item  We define  a compact ball $B_r\subset U$ in the linearized neighborhood of $\sigma_r$,such that 
if $ (x,y,z)\in B_r$ then $y\leq0$ and it is $0$ only in $p'\in W_X^s(\sigma_r)$. 
The ball $B_r$  and $D_r$ can be taken so that there exist $t_0<K_Y$ such that $\phi_X^{t}(B_r)\subset D_r$  for all $t>t_0$.
 Note that here the stable and unstable manifolds refer to de dynamics of $X$.

\end{itemize}
\end{defi}
\begin{figure}[htb]
\begin{center}
\includegraphics[width=0.60\linewidth]{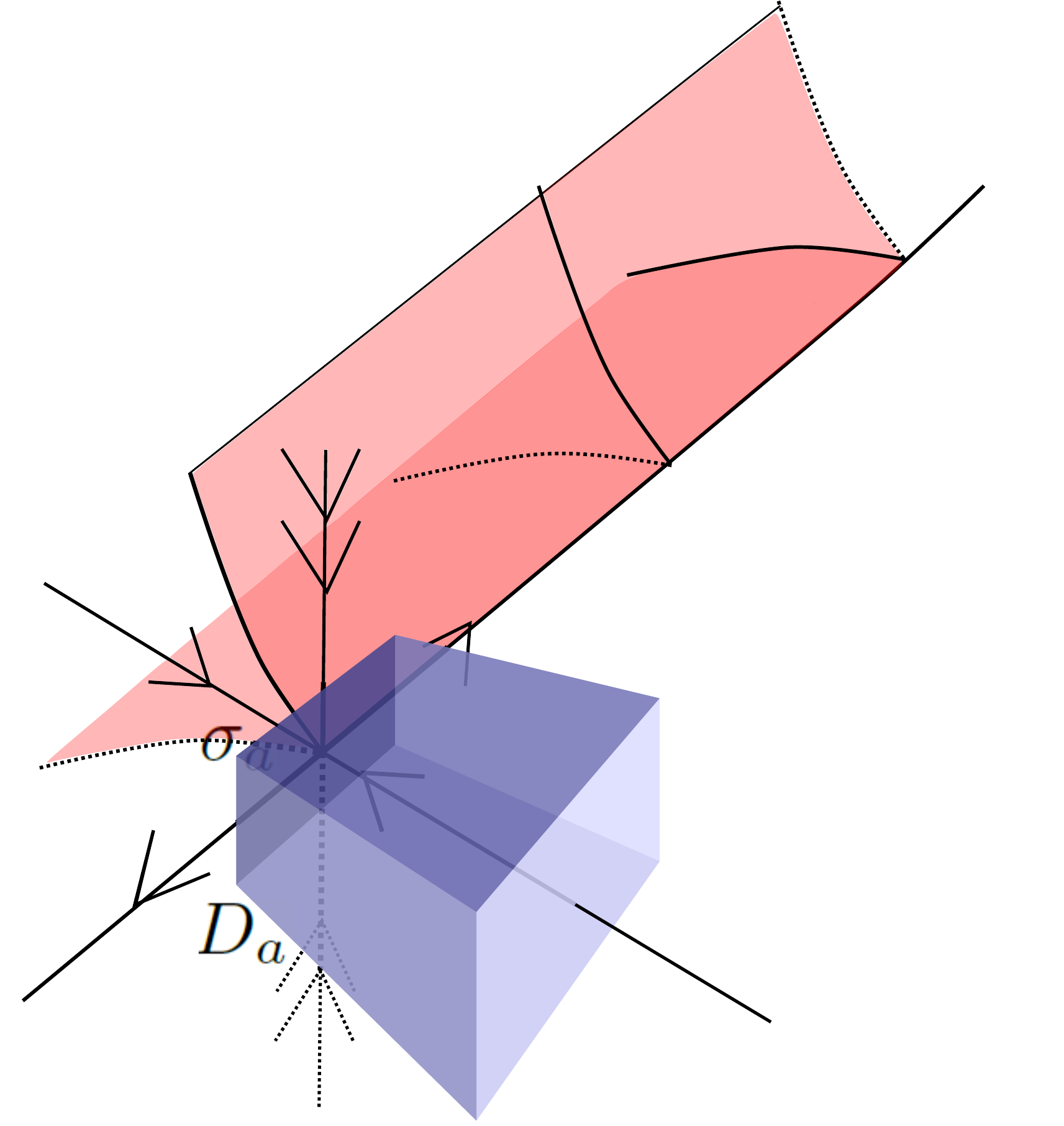}
\end{center}
\caption{The ball $D_a$.}
\end{figure}

 \subsection{Choosing the right isotopy}\label{s.isotopy}
  We  now
choose some more properties on the diffeomorphism $H$ from proposition \ref{p.max1}.
The following Lemma will be proven in subsection \ref{s.conH}.
\begin{lemm}\label{l.propH}
Let $\Sigma\subset S^3 \times \mathbb{RP}^2$ be $\sum=S^3\times T$ where $T\subset  \mathbb{RP}^2$ is a uni-dimensional  transverse section to the vector field $Y$ parametrized by  $[-1,1]$.
There exist a $\mathcal{C}^{\infty}$ diffeomorphism isotopic to identity,  $H:\Sigma\to\Sigma$,
$H(x,l)=(r_l(x),\theta_x(l))$
 where $r_l: \sum\to S^3$, $\theta_x:\sum\to T $, with  following properties:
\begin{itemize}

\item $H$ is the identity on the boundary of $\sum$  and  outside of $V_a$ and in $V_r$.
\item The map $r_l(x)$ is the identity for $l=1$ or  $l=-1$, or if $x\in \overline{D_r}$.
\item 
The image of $r_l(B_a)=B_r$, and $r_l(p)=p'$ for all $l\in[-1/2,1/2]$.
\item  If $l\in[-1/2,1/2]^c$ then $\theta_x(l)=l$.
\item  If $l\in[0,1/2]$ and $x\notin B_a$ $\theta_x(l)>0$ .
\item  If $l\in[0,1/2]$ and $x\in B_a$ then $-\epsilon\leq\theta_x(l)\leq \epsilon$,
\item  The only point $l$ such that $H(p,l)=(p',0)$, is $l=0$.
\item $H(S^3\times\set{0})\times\set{0}\in\Sigma\times\set{0}$$  cuts transversally $$S^3\times\set{0}\times\set{0}\in\Sigma\times\set{0}\,.$
\end{itemize}
\end{lemm}

\begin{figure}[htb]
\begin{center}
\includegraphics[width=0.65\linewidth]{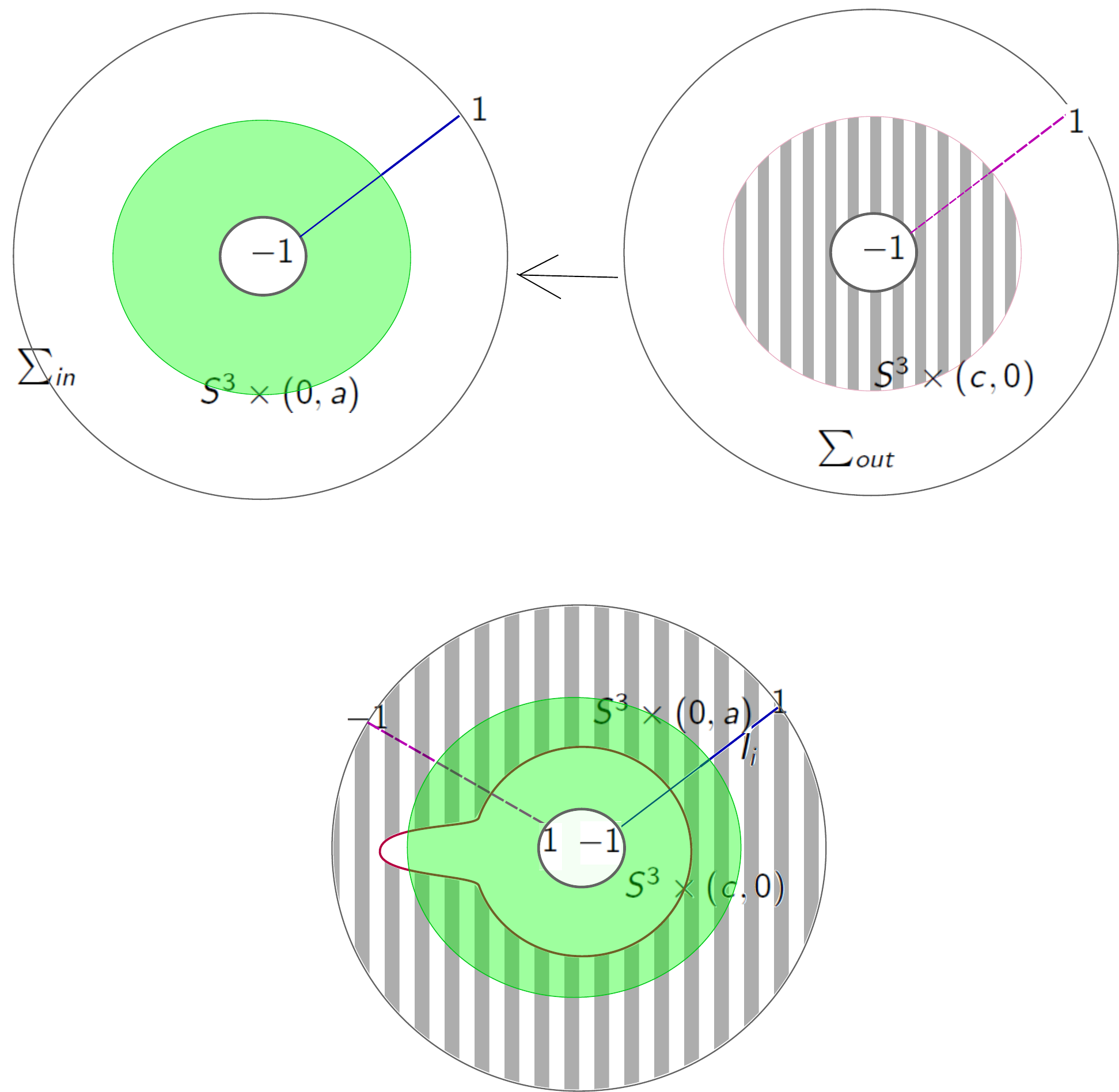}
\end{center}
\caption{The function $\theta_x(l)$. Note that the painted area depicts orbits with their $\alpha$ limits in $U_{\alpha}$ and the striped area depicts orbits with their $\omega$ limits in $U_{\omega}$}
\end{figure}


\begin{figure}[htb]
\begin{center}
\includegraphics[width=0.70\linewidth]{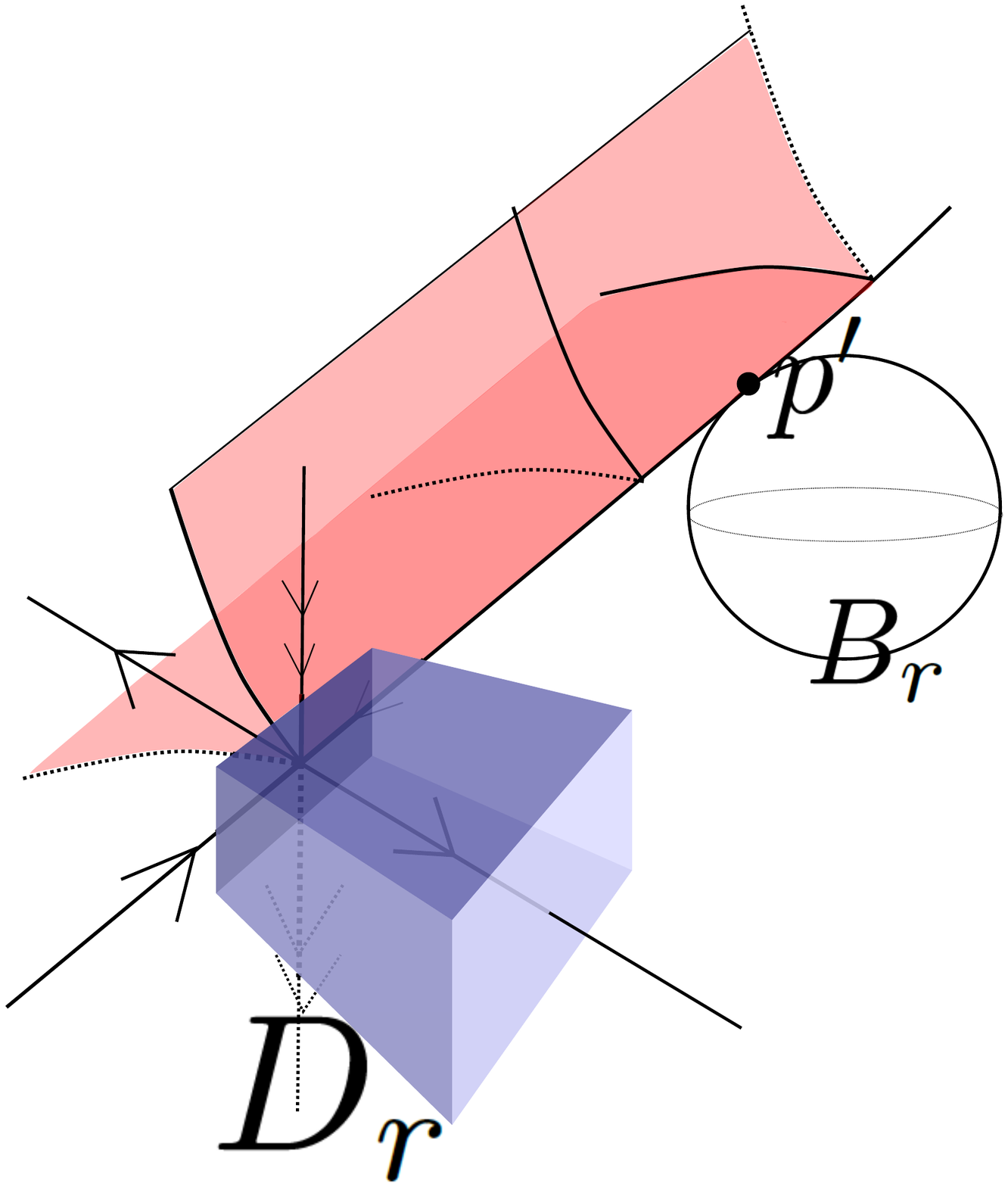}
\end{center}
\caption{The ball $B_a$}
\end{figure}

\begin{prop}\label{p.max2}
We consider $H$ as in Lemma \ref{l.propH}, then the orbits in the maximal invariant set $\Lambda_H$ are contained in $L_a\times \set{s}$ or cross the flow box $\sum\times [-1,0]$ in $$B_a\times [0,1/2]\times\set{-1}\,.$$

\end{prop}
\proof
Suppose that $\gamma $ is an orbit in $\Lambda_H$ that doesn't cross $\sum$. From Proposition $\ref{p.max1}$ these orbits of $\Lambda_{H}$ are in $S^3\times\set{s}$.
Let $h $  be a point of $\gamma$ of coordinates $(x,j)\in S^3\times\mathbb{RP}^2$. If $x$ is not in $\La$ (for $X$) then, the alpha or the omega limit of $x$ must be in $U^c$. Therefore, for a $t$ large enough, $\phi^t_{H}(x,j)\notin V$ ( Recall that $V=[U\times \mathbb{RP}^2]/\set{U_{\alpha}\cup U_{\omega}}\,$).

Then if $\gamma$  doesn't intersect $\sum$, it must be in $\La\times \set{s}$.

Let us suppose now that $\gamma $ intersects $\sum$. Let $h$ be a point in $\gamma\cap\sum\times[-1,0]$ such that $h\in \sum\times\set{-1}$. We write $h$ as $(y,-1)$ and $y$ as $y=(x,l)\in\sum$.
\begin{enumerate}
\item If $l>1/2$, or if $x\notin B_a$ with $l>0$, then $H(x,l)=(r_l(x),\theta(l))$ with $\theta(l)>0$ and then $\phi^{1}_H(y)=(r_l(x),\theta(l))\times\set{0}$. Since outside of the flow-box $Z_{id}=Z_{H}$ now we can look at $Z_{id}$. From the properties of  $Y$ (\ref{propY}) we have that the future orbit of  $\theta(l)>0$, does not cross $T$  and the $\omega$-limit   is $\omega$.
    Then the orbit for $\phi^{t}_H$ is in $U_{\omega}$ for a large enough $t$.
    Then $\gamma$ is not in $\Lambda_H $.
\item  If $l<0$, since $h$ goes  outside of the flow-box for the past (where $Z_{id}=Z_{H}$) now we can look at $Z_{id}$. From the properties of  $Y$ (\ref{propY}) we have that the orbit of  $l<0$ does not cross $T$ for the past and the $\alpha$-limit  is $\alpha$. Then $\gamma$ does not cross again the flow-box for the past. The orbit for $\phi^{t}_H(h)$ is in $U_{\alpha}$ for a negatively large enough $t$ and $\gamma$ is not in $\Lambda_H $.

\item If $x$ is not in  $B_a$ and $l=0$ then $\phi^{1}_H(h)=((H(x),\theta_x(l)),0)$ and $\theta_x(l)>0$. Then, as before,  we have that the orbit of $\theta_x(l)$ for $Y$ does not cross $T$ for the future and the $\omega$-limit  for $Y$ is $\omega$. Then $\gamma$ does not cross again the flow-box for the future and $\gamma$ is not in $\Lambda_H $.

 \end{enumerate}
Then the only other case in which $\gamma$ might be in $\Lambda_H $ is if $\gamma$ crosses  the flow box $\sum\times[-1,0]$ in $B_a\times [0,1/2]\times\set{-1}$.
 \endproof

\begin{prop}\label{p.max3}
There is a unique orbit $\gamma$ in $\Lambda_H$ that crosses $\sum$, that orbit is the orbit of $$(p,0)\times \set{-1}\in \sum\times\set{-1}\,.$$
\end{prop}
\proof
Let $\gamma$ be an orbit in $\Lambda_H$.
From proposition (\ref{p.max2}), we already know that if an orbit of $\Lambda_H$ crosses $\sum$ then it crosses at a point $y=(x,l,-1)\in B_a\times [0,1/2]\times\set{-1}$.

If $\theta_x(l)<-\epsilon<0$ recall that the properties of $H$ (Lemma \ref{l.propH}) give us that  then $l<0$.

Suppose now that $l\geq0$ and that $\theta_x(l)>0$. As in our previous proposition this implies that $\phi_H^{t+1}(y)\notin U$ for $t$ large enough.

If $l\geq0$ and $-\epsilon\leq\theta_x(l)\leq 0$  then $x\in B_a$, and  $r_l(x)\in B_r$. Suppose that $x\neq p$.
 Then $$\phi_H^{t+1}(y)\notin \sum\times[-1,0]$$ for all $K_Y>t>0$, and therefore $\phi_H^{t+1}(y)=\phi_{id}^{t}(\phi_H^{1}(y))$ for all $K_Y>t>0$.
 Let us consider $t_0$ as in the properties of $H$ (Lemma \ref{l.propH}). Recall that $t_0 $ is such that  $\phi_X^{t}(B_r)\subset\overline{D_r}$  for all $K_Y>t>t_0$.
We  call  $$\phi_{id}^{t_0}(\phi_H^{1}(y))=(x_1,z_1)\in S^3\times\mathbb{RP}^2\,.$$
 Since $x_1\in u_r$ and is not in $L_r$ (since $x_1\neq p'$), then   $x_1$ is in the  attracting region of a sink  $p_r$, $D_r$  of $X$ (see subsection \ref{completeflow3}).
 Now, for all $t>t_0$ even the ones bigger than $K_Y$,  we have that  $$\phi_H^{t}(x_1,z_1)=(\phi_{Hx}^{t}(x_1,z_1),\phi_{Hz}^{t}(x_1,z_1))$$ and since every time for the future that this orbit crosses the flow-box $\sum\times[-1,0]$, the function $r_l$ is the identity (since  $r_l$ is the identity in $D_r$ ), then $$\phi_H^{t}(x_1,z_1)=(\phi_{X}^{t}(x_1),\phi_{Hz}^{}(x_1,z_1))\,.$$

 Since
 $\phi_X^{t}(x_1)\notin U$ for $t>t_0$ big enough,  then  $\phi_H^{t'}(y) $ is eventually not in $V$ for some   $t>t_0$. Then  $\gamma$ is not in $\Lambda_H$ as wanted.

If $l\geq0$ and $-\epsilon\leq\theta_x(l)< 0$ but $x=p$.
Let $t_y$ be a time in which the orbit returns to the flow-box. That is $t_y$ is such that $$\phi_H^{t_y}(y)=(x_1,l_1,-1)\in \sum \times\set{-1}\,.$$
  Recall from  the properties of $H$ (Lemma \ref{l.propH}) that $t_y\geq K_Y >t_0$ with $t_0$  such that  $\phi_X^{t}(B_r)\subset \overline{D_r}$  for all $t>t_0$.
 Since$-\epsilon\leq\theta_x(l)<0$ and after returning to $\sum \times\set{-1}\,$ the orientation was reversed, then $l_1$ is positive. Since now $x_1$ is not in $B_a$, then $\theta_{x_1}(l_1)>0$. So, now for a small $t>0$, we have that  $$\phi_H^{t+t_y+1}(y)=\phi_{id}^{t}(\phi_H^{t_y+1}(y))\,.$$ This implies that the orbit of $y $ never cuts the flow box again, and therefore, for a big enough $t$, $\phi_H^{t}(y)$ is in $U_{\omega}$. As a consequence  $\gamma$ is not in $\Lambda_H$ as wanted.

The only case left is  $x=p$ and $\theta_x(l)=0$. The last property of $H$ (Lemma \ref{l.propH}) tells us that  $l=0$, so the objective now is to prove that the orbit of
$y=(p,0) $ never leaves $$V=[U\times \mathbb{RP}^2]/\set{U_{\alpha}\cup U_{\omega}}\,.$$
But  $(p,0) $ is in the stable manifold of $\sigma_r$ and in the unstable manifold of $\sigma_a$, and then then the orbit of $y$ is in $\La_H$..

\endproof

\subsection{Multisingular hyperbolicity}
Until now we have constructed a vector field having a chain recurrent class such that  \begin{itemize}
                                                       \item Two singularities of different indexes one in $L_a$ and the other in $L_r$.
                                                       \item All the periodic orbits have the same index and the singularities are in the closure of the periodic orbits.
                                                       \item There are  periodic orbits in $L_a$ such that their  stable manifolds intersect  the unstable manifolds of periodic orbits in $L_r$.
                                                       \item There is only one orbit in the class with the $\alpha$-limit in $L_a$ and the $\omega$-limit in $L_r$.
                                                     \end{itemize}
The goal now is to show that we can choose a diffeomorphism $H $ so that this vector field would be strong multisingular hyperbolic.
After that we will perturb this vector field to an other that will still be strong multisingular hyperbolic, but having a homoclinic connection between periodic orbits in $L_a$ and periodic orbits in $L_r$. This will finish the proof of theorem \ref{t.example}.

In subsection \ref{s.conH} we will prove that there exist a diffeomorphism $H $ with the properties defined in  Lemma \ref{l.propH}, in particular   with the following  property, we require that $$H(S^3\times\set{0})\times\set{0}\in\Sigma\times\set{0}$$  cuts transversally $$S^3\times\set{0}\times\set{0}\in\Sigma\times\set{0}\,.$$

This last property guaranties that the set  $\Lambda_H$  will be strong multisingular hyperbolic.

To show that $\Lambda_H$  is strong multisingular hyperbolic we need to check that we are in the hypothesis of Lemma \ref{c.unamashyp}. Since we have already shown the other hypothesis
the following Lemma implies strong multisingular hyperbolicity.
\begin{lemm}\label{l.unaorbitamas}
Let $y\in\Sigma\times\set{-1}$ be such that $\Lambda_H=\Lambda\cup O(y)$.
There exist a diffeomorphism $H$ such that \begin{itemize}
                                             \item The stable and unstable spaces along the orbit of $S_{Z_H}(y)$ intersect transversally,
                                             \item The orbit of $y$ does not intersect the escaping spaces of the singularities for $Z_H$,
                                           \end{itemize}
Moreover this implies that  $\Lambda_H$ is multisingular hyperbolic.

\end{lemm}

\proof
Consider the points in $S^3\times \mathbb{RP}^2$, $a=(p,s)$ and $b=(p',s)$ and  $y=(p,0,-1)\in\Sigma\times \set{-1}$.
The orbit of $y$ is in the strong unstable manifold of $a$, (since  the unstable manifold of $s$ intersects $T$ at $0$ for $Y$ ). Analogously $y$ is in the strong stable manifold of $b$ since $\phi_H^{1}(y)=(p',0,0)$. Observe that $a$ and $b$ are regular points and $p\in W^u(\sigma_a)$ and $p'\in W^s(\sigma_r)$ for $X$. therefore $\gamma$ does not intersect the escaping spaces of the singularities for $Z_H$. From Proposition \ref{p.unamasextended} this implies that the center space of the singularities of  $\Lambda_H$ and $\Lambda$ are the same.

From Lemma \ref{l.masuna} we have that there exists an unstable space (for the reparametrized linear Poincar\'{e} flow ) at $y$ that we call $E^u_y$. We can consider a metric so that $\Sigma$ is always  normal 
to the vector field at the flow box.
We take a vector $v\in E^u_y$ at $y$. This vector is tangent to  $$S^3\times\set{0}\times\set{-1}\in\Sigma\times\set{-1}$$ at $y$.

Let us recall that we have assumed at the beginning of the subsection that the image of $$H(S^3\times\set{0})\times\set{0}\in\Sigma\times\set{0}$$ under $H$ cuts transversely $$S^3\times\set{0}\times\set{0}\in\Sigma\times\set{0}\,.$$ Then the image of $v$ under the differential of $H$ (and of $\phi_H^{1}$) is transverse to $S^3\times 0\times O(y)$ at $\phi_H^{1}(y)$, and then so is the image of $v$ under $\Psi^1(v)$, since the direction of the flow is perpendicular to $T\times\set{0}$.
On the other hand Lemma \ref{l.masuna} also gives us a  stable space $E^s_y$ at  $\phi_H^{1}(y)$  that is tangent to  $S^3\times \set{0}\times\set{0} $ at $\phi_H^{1}(y)$. Then the stable and unstable spaces of the reparametrized linear Poincar\'{e} flow are transverse. Then  we are in the hypothesis of Corolary \ref{c.unamashyp} and this completes the proof.
  \endproof


 With this last Lemma we know that the maximal invariant set $\La_H$ is multisingular hyperbolic. But this is not enough, since a small perturbation of $Z_H$ could brake the connection between $L_a$ and $L_r$ and have $\sigma_a$ and $\sigma_r$ in different chain classes.
 We need now to show that the right perturbation of   $Z_H$  will generate  intersection of the stable and unstable manifolds of periodic orbits in $L_a$ and $L_r$. Since $\La_H$ is multisingular hyperbolic for $Z_H$, so will it be for this new vector field and now the singularities will be robustly in the same chain recurrence class.

  The following Lemma implies Theorem \ref{t.example}
  \begin{figure}[htb]
\begin{center}
\includegraphics[width=0.70\linewidth]{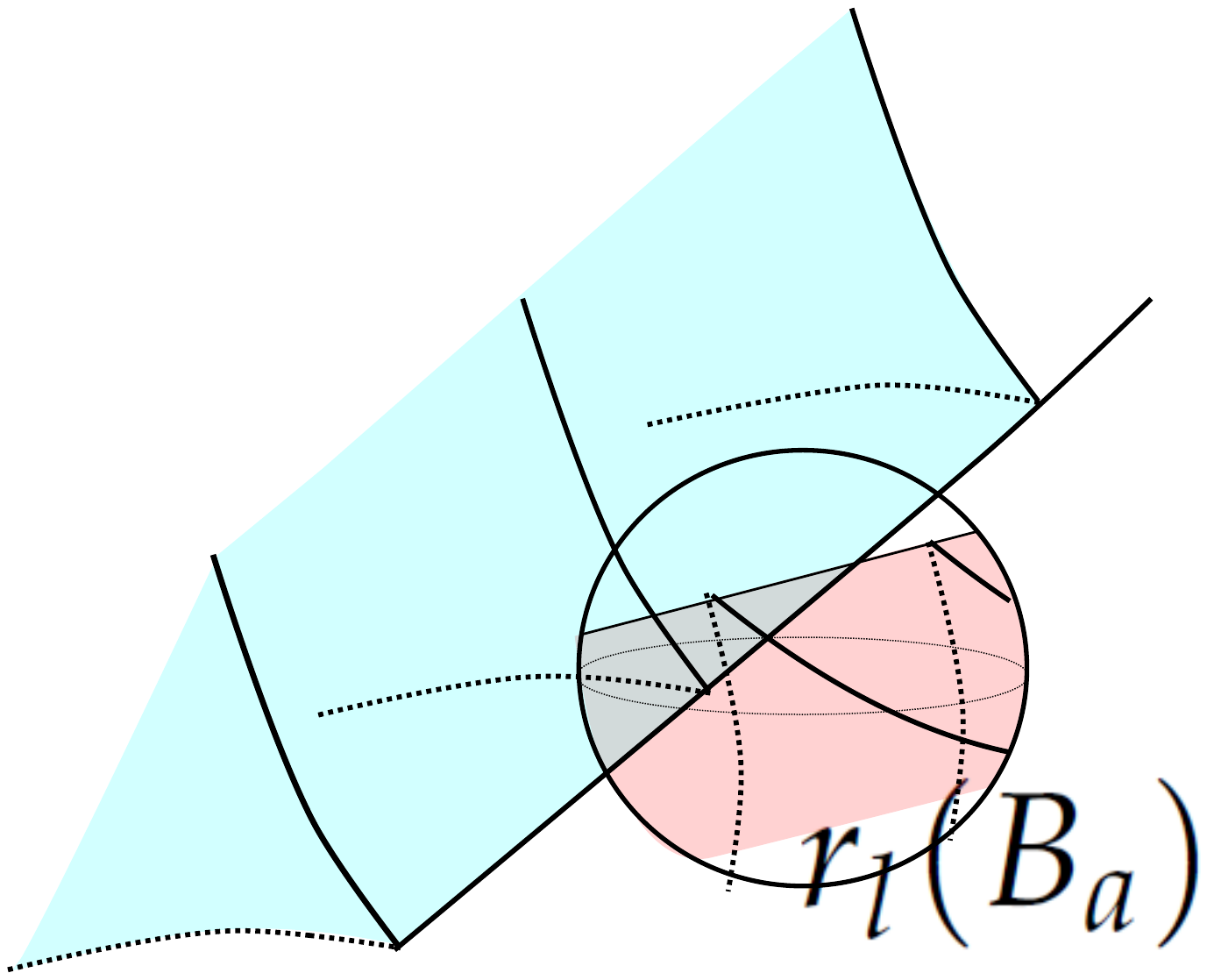}
\end{center}
\caption{A perturbation of $Z_H$, in particular of $r_l(x)$.}
\end{figure}
\begin{lemm}\label{perint}
  There is an arbitrarily small perturbation of $Z_H$, that we call $Z_{H_{\epsilon}}$, and a $C^1$ neighborhood of $Z_{H_{\epsilon}}$ called $\cV$ so that any vector field $Z\in\cV$ has a maximal invariant set $\La_Z$ that is multisingular hyperbolic and there is a chain class $C\subset\La_Z$ that has two singularities of different index accumulated by periodic orbits.

\end{lemm}

\proof
We will make a small perturbation of $H$ and this will result in a small perturbation of $Z_H$.
Let us recall that we can write $H$ as
$$H(x,l)=(r_l(x),\theta_x(l))$$ where $r_l: \sum\to S^3$, $\theta_x:\sum\to T $.

Let us again choose coordinates in the linearized neighborhood of $\sigma_r$ as the chosen before. Consider a small positive traslation along the $y$ direction of the ball $B_r$, $T_\epsilon(B_r)$ so that there exists  $b_r\subset{B_r}$  a small ball such that the periodic orbits of $L_r $ are dense in $b_r$ and has $T_\epsilon (p')$ in it's boundary. Analogously for $B_a$
We can perturb $r_l$ to $r'_l(x)$ so that $T_\epsilon(B_a)$ is sent to $T_\epsilon(B_r)$ and  $r'_l(T_\epsilon(p))$ is now sent to $T_\epsilon(p')$. 
By choosing $\epsilon$ small enough this can be achived with a $C^r$ small perturbation of $r_l$.
The resulting vector field  $Z_{H_{\epsilon}}$ is still $C^1$ and multisingular hyperbolic.


Now from the fact that periodic orbits are dense in the sets $L_a$ and $L_r$, and the fact that  $Z_{H_{\epsilon}}$ is star, we get that we can choose a small perturbation by the connecting Lemma in [\cite{BC}] so that the  some periodic  orbit in $L_a$ is homoclinically related to a periodic orbit in $L_r$. Recall that the periodic orbits all have the same index.
This homoclinic intersection is roust.
\endproof

\section{Appendix: Sub constructions}
\subsection{ Construction of the vector field $Y$ in $\mathbb{RP}^2$ }\label{ss.propY}

In this subsection we prove the following Lemma from  section (\ref{s.main})
\begin{lemm*}
There exist a vector field $Y$ in $\mathbb{RP}^2$ with the following properties:
\begin{itemize}
\item $Y$ is a $\mathcal{C}^{\infty}$ vector field

\item It has 3 singularities: a saddle singularity $s$, a source $\alpha$ and a sink $\omega$. 
\item The contracting and expanding  Lyapunov exponents of the saddle are equal in absolute value ($\lambda_{sss}=-\lambda_{uuu}$), and $\lambda_{uuu}>> 6$.
\item One of the stable branches of $s$ (that is an orbit) has $\alpha$-limit  $\alpha$,
and one of the unstable branches of $s$ (that is an orbit) has $\omega$-limit   $\omega$.
\item The other two branches form an  orbit with $\alpha$-limit and $\omega$-limit in $s$ and we call this orbit $\gamma$.
\item There is a transverse section to $\gamma$ and to the flow, that we call $T=[-1,1]\times\set{a}$,  $T\cap\gamma=0\times{a}$ and the flow of $Y$, $\phi^Y(s,t)$ is such that:
    \begin{itemize}
    \item  If $s=(x,a)\in T$ is such that $x>0$, then $\phi^Y(s,t)$  does not cross $T$  for any $t>0$ and has $\omega$-limit in $\omega$. And for $t<0$ there exists only one $t_s<0$ such that $\phi^Y(s,t_s)=s'\in T$ with $s'=(x',a)$, $x'<0$ and the $\alpha$-limit of $s$ is $\alpha$.
    \item If $s=(x,a)\in T$  and $x<0$, then  $\phi^Y(s,t)$ does not cross $T$  for any $t<0$ and has $\alpha$-limit in $\alpha$. And for $t>0$ there exists only one $t_s>0$ such that $\phi^Y(s,t_s)=s'\in T$ with $s'=(x',a)$, $x'>0$ and the $\omega$-limit of $s$ is $\omega$.

    \end{itemize}
\end{itemize}
\end{lemm*}
 \subsubsection{ A vector field with a saddle connection in a M\"{o}bius strip }

 Let us start by defining some simple linear flow in $\RR^2$.
 We take a linear vector field $Y(x,y)=(\lambda_{sss}x, \lambda_{uuu}y )$ defined in $[-2,2]\times[-2,2]$. We ask that $\lambda_{uuu}=-\lambda_{sss}$ and we also ask that  $\lambda_{uuu}>6$.



We consider a close curve $C$ formed by the union of following curves:
\begin{itemize}
\item We consider the orbit of  a point $(-a,2)$. This orbit cuts the vertical line $(-2,y)$ in a point $(-2,a')$. The segment of orbit from $(-2,a')$ to $(-a,2)$ is our first curve $C_1$.
\item  We consider the orbit of  a point $(a,2)$. This orbit cuts the vertical line  $(2,y)$ in a point $(2,c)$. The segment of orbit from $(a,2)$ to $(2,c)$ is  $C_2$.
\item We consider the segment $\set{-2}\times[a',-a']$ as our second curve  $C_3$.

\item We take the  orbit of $(-2,-a')$ and we call the point where it cuts the horizontal line $l$ in a point $(-b,-2)$. The segment of orbit from $(-2,-a')$ to $(-b,-2)$ is our third curve  $C_4$.
\item We consider the segment $\set{2}\times[-c,c]$ as our second curve  $C_5$.

\item  We consider the orbit of  a point $(2,-c)$. This orbit cuts the horizontal line  $(x,-2)$ in a point $(b',-2)$. The segment of orbit from $(2,-c)$ to $(b',-2)$ is  $C_6$.
\item The segment  $[b',-b]\times\set{-2}$ our forth curve  $C_7$.

  \item The segment  $[-a,a]\times\set{2}$ our last curve  $C_8$.
\end{itemize}
\begin{figure}[htb]
\begin{center}
\includegraphics[width=0.70\linewidth]{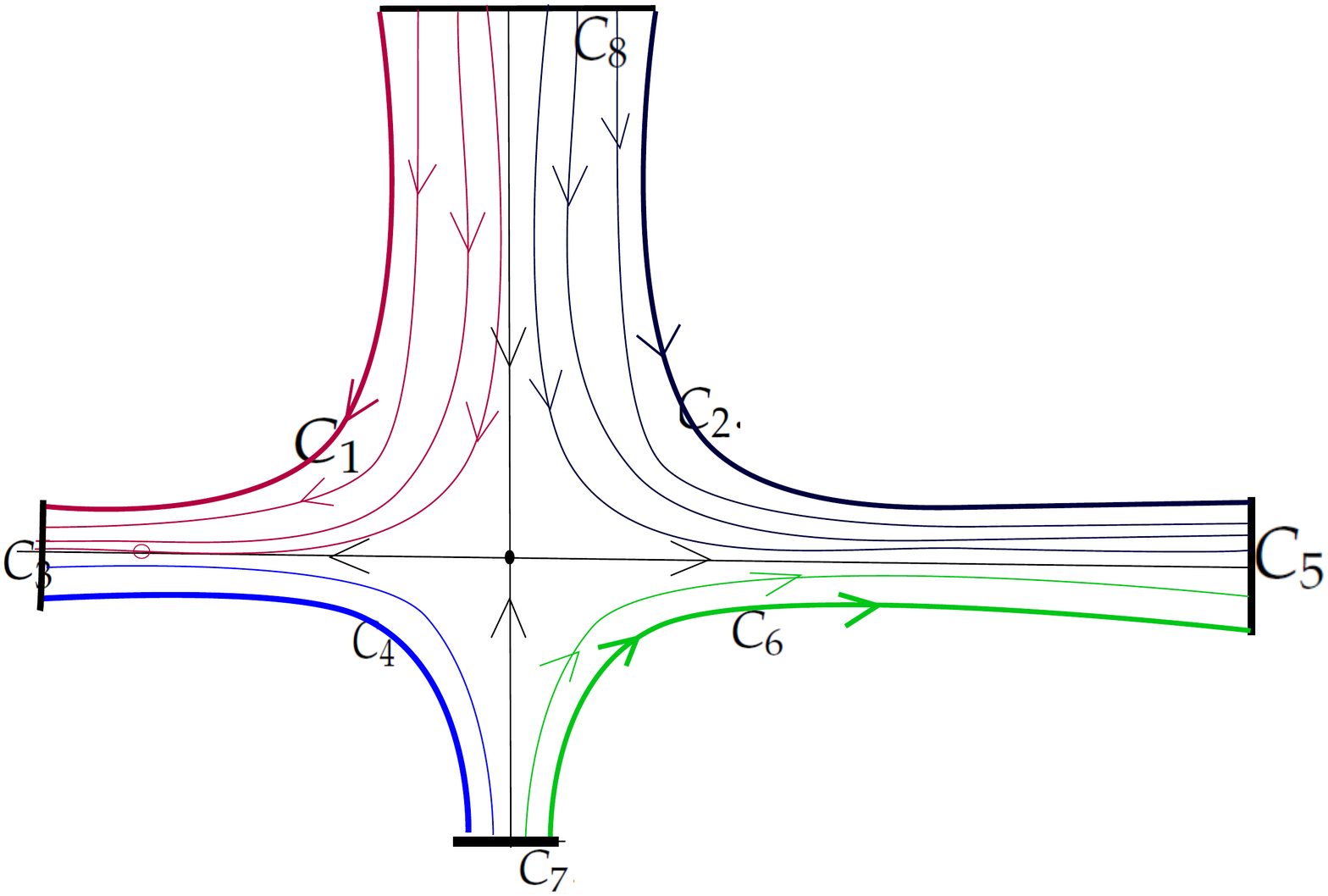}
\end{center}
\caption{The vector field $Y$ in $D$.}
\end{figure}

There is a diffeomorphism that reverse orientation $d:C_8\to C_5$ defined as follows: $$d(x)=-\frac{c.x}{a}\,.$$
Now we glue $C_8$ and $C_5$ along $d$.
There is a connected component in the complement of $C$ that contains $(0,0)$. We call the closure of this connected component $D$.
The manifold $D$ (with boundary $C$) obtained from this gluing is a 2 dimensional non-orientable manifold with a connected boundary, therefore it is a M\"{o}bius strip.

Note that since the $d:C_8\to C_5$ is such that 0 is maped to 0,  then there is a branch of the stable manifold of $(0,0)$ and a branch of the unstable manifold of $(0,0)$ that intersect. That is, there is an orbit $\gamma$ such that $$\gamma\subset W^s(0,0)\cap W^u(0,0)\,.$$ We say then that $(0,0)$ has a saddle connection.

 \subsubsection{Completing the vector field to $\mathbb{RP}^2$}

Let us consider a linear vector field in $\RR^2$ with a sink $\omega$, and let us take a neighborhood $u_{\omega}$ in its basin of attraction. We choose a curve in the boundary such that the vector field is  pointing inwards.
Now we can glue $u_{\omega}$ to $D$ Identifying an arc in  the boundary of $u_{\omega}$ with  $C_3$ along a diffeomorphism,  since the vector field $Y$  is pointing outwards at $C_3$.

 Note that the unstable orbit that is not in the saddle connection in $D$ now has its $\omega  $-limit in $\omega$, after the gluing.

 We call the new vector field $Y$ and what remains of the boundary of  $u_{\omega}$, we now call it $C_3'$.

 Analogously we attach a neighborhood $u_{\alpha}$, containing a source $\alpha$ 
 and glue it to $D$ by identifying a curve in the boundary of   $u_{\alpha}$ witht the segment $C_7$. We call the subset of boundary of $u_{\alpha}$, that was not glued to $D$,  $C_7'$.

Note that the stable orbit  of $(0,0)$ that is not in the saddle connection in $D$ now has its $\alpha$-limit  in $\alpha$.

 We call $D'$ to the region formed by $D$ with $u_{\alpha}$ and $u_{\omega}$ attached.
Since $D'$ is a M\"{o}bius strip, then the complement in
 $\mathbb{RP}^2$ is a disc $R$ having a boundary formed by 4 disjoint curves tangent to the flow ($C_1$, $C_2$, $C_6$ and $C^4$), one curve transverse to the flow and entering $D'$ $C_7'$, and one curve transverse to the flow and exiting  $D'$ $C_3'$. Therefore we can define the flow in the complement of $D'$ in the trivial way by sending the points in  $C_3'$ to  $C_7'$.

 Now we prove Lemma \ref{propY}
\proof
\begin{itemize}
\item Since the original maps are linear, the resulting map after the gluing is also $\mathcal{C}^{\infty}$.
\item The contracting and expanding  Lyapunov values of $Y$ can be taken to be as strong as required
\item As noted above,  one branch of each stable and unstable manifold form a saddle connection $\gamma$ while the others come or go to the sink and source.
\item  The segment, $T_0=C_5$ is a transverse section to $\gamma$ by construction and    is such that:\begin{itemize}
    \item  If $s>0$ $\phi^Y(s,t)$  never touches $T_0$  for any $t>0$ and has $\omega$-limit in $\omega$. And for $t<0$ there exists only one $t_s<0$ such that $\phi^Y(s,t_s)=s'\in T_0$ with $s'<0$. and the $\alpha$-limit of $s$ is $\alpha$.
    \item If $s<0$ $\phi^Y(s,t)$  never touches $T_0$  for any $t<0$ and has $\alpha$-limit in $\alpha$. And for $t>0$ there exists only one $t_s>0$ such that $\phi^Y(s,t_s)=s'\in T_0$ with $s'>0$ and the $\omega$-limit of $s$ is $\omega$.

\end{itemize}
As a consequence of the fact that that $C_8$ was glued to $C_5$ reverting orientation.
\end{itemize}
\endproof
\subsection{Construction of the diffeomorphism $H$}\label{s.conH}
In this subsection we prove the following Lemma from subsection (\ref{s.isotopy}) :
\begin{lemm*}(\ref{l.propH})
Let $\Sigma\subset S^3 \times \mathbb{RP}^2$ be $\sum=S^3\times T$ where $T\subset  \mathbb{RP}^2$ is a uni-dimensional  transverse section to the vector field $Y$ parametrized by  $[-1,1]$.
There exist a $\mathcal{C}^{\infty}$ diffeomorphism isotopic to identity,  $H:\Sigma\to\Sigma$,
$H(x,l)=(r_l(x),\theta_x(l))$
 where $r_l: \sum\to S^3$, $\theta_x:\sum\to T $, with  following properties:
\begin{itemize}

\item $H$ is the identity on the boundary of $\sum$  and  outside of $V_a$ and in $V_r$.
\item The map $r_l(x)$ is the identity for $l=1$ or  $l=-1$, or if $x\in \overline{D_r}$.
\item 
The image of $r_l(B_a)=B_r$, and $r_l(p)=p'$ for all $l\in[-1/2,1/2]$.

\item  If $l\in[-1/2,1/2]^c$ then $\theta_x(l)=l$.
\item  If $l\in[0,1/2]$ and $x\notin B_a$ $\theta_x(l)>0$ .
\item  If $l\in[0,1/2]$ and $x\in B_a$ then $-\epsilon\leq\theta_x(l)\leq \epsilon$,
\item  The only point $l$ such that $H(p,l)=(p',0)$, is $l=0$.
\item $H(S^3\times\set{0})\times\set{0}\in\Sigma\times\set{0}$  cuts transversally $S^3\times\set{0}\times\set{0}\in\Sigma\times\set{0}\,.$

\end{itemize}
\end{lemm*}

\proof

Let us consider  a  closed neighborhood of $D_a\cup D_r$ that we call $C$.
Since $B_a$ and $B_r$ are subsets of $S^3$, they are isotopic to each other. Moreover,  we can choose  $C$   so that they are isotopic to each other in $S^3\setminus C$, since $D_a\cup D_r$ does not disconnect $S^3$.
Therefore there is a function $r':S^3\setminus C^{\mathrm{o}} \times[0,1/2]\to S^3$ such that $$r'(x,0)=id(x) \text{ and }r(B_a,1/2)=B_r.$$
We can choose $r'$ so that it is the identity in the boundary of $C$ for all $t\in[0,1/2]$ and such that $r'(p,1/2)=p'$
We can extend now this function to $S^3$ by asking that $r'\mid{D_a\cup D_r}=Id$.
Now $r:S^3\times[-1,1]\to S^3$ is defined by
\begin{equation}
    r(x,l)=\left\{
\begin{array}{ll}
r'(x,l+1), & \mathrm{ } \text{ if }  l\leq -1/2 \\
r'(x,1/2), & \mathrm{  } \text{ if }  -1/2<l\leq 1/2\\
r'(x,l-1), & \mathrm{ } \text{ if }  l>1/2\,.
\end{array}
\right.
\end{equation}
Now we need to construct $\theta:\sum\to[-1,1]$.

We consider a $\mathcal{C}^{\infty}$ bump function $h:[-1,1]\to[-1,1]$,
\begin{itemize}
\item  If  $l\in[-1/2,1/2]^c$ then $h(l)=0$,
\item if   $l\in[-1/2,1/2]$ then $0<h(l)\leq\frac{\epsilon}{2}$,
\item if   $l=0$ then $h(l)=\frac{\epsilon}{4}$,
\item  $\frac{\partial h(l)}{\partial(l)}\mid_{(0)}\neq 0$
\item  $\frac{\partial h(l)}{\partial(l)}< 1$ for all $l$
\end{itemize}
We can also assume that $h$ is sufficiently differentiable.
Let $B_A$ be an arbitrarily small neighborhood of $B_a$.
We consider now a second bump function $g:S^3\to[-1,1]$
\begin{itemize}
\item  If  $x\in B_A^c$ then $g(x)=0$,
\item If  $x\in B_a$ then $\epsilon\leq g(x)\leq\frac{\epsilon}{2}$,
\item  $g(p)=-\frac{\epsilon}{4}$ and $\frac{g(x)}{\partial(v)}\neq 0$ for any given $v$ direction in $S^3$.  
\end{itemize}
 We define then  $\theta_x$ as follows:
$$\theta_x(l)=id(l)+h(l)+g(x)\,.$$
Since $\frac{\partial h(l)}{\partial(l)}< 1$ for all $l$, $\theta_p(l)=id(l)+h(l)+g(p)\,$ is an increasing function that is $0 $ only at $l=0$.

Note that the image of the  vectors tangent to the coordinates in $S^3\times \set{0}$, under the differential of $H$, have a non vanishing component in the direction of $T$.
This is our desired function.
\endproof
\subsection{ Construction of the plug}\label{plug}

The aim of this subsection is to  prove the following theorem:

\begin{theo*}[\ref{t.tube}]
There exist a vector field $\chi$ such that its flow  $\phi_{\chi}$  defined  in $S^3$
 has  the following properties:

 There is a region $S^2\times{[-1,1]}\subset S^3$ such that
\begin{itemize}
\item The  vector field $\chi$ is entering at $S^2\times\set{-1}$ and points out at $S^2\times\set{1}$
\item The  vector field $\chi$ is such that the chain recurrent set consists of 2  sources  singularities, $p_1$ and  $p_1'$,  2 sinks singularities, $p_2$ and $p_2'$, and  2 periodic saddles, $p_3$ and $p_3'$.
 \item The intersection of the invariant manifolds of the saddles,  with the boundary of $S^2\times{[-1,1]}$, are disjoint circles that we name as follows:
  \begin{itemize}
  \item $W^{s}(p_3)\cap S^2\times{[-1,1]} = c_0$ in $S^2\times\set{-1}$,
  \item $W^{s}(p'_3)\cap S^2\times{[-1,1]} = c'_0$ in $S^2\times\set{-1}$,
\item $W^{u}(p_3)\cap S^2\times{[-1,1]} =c_1$ in $S^2\times\set{1}$,
\item  $W^{u}(p'_3)\cap S^2\times{[-1,1]} = c'_1$ in $S^2\times\set{1}$.
 \end{itemize}
\item The circle $c_0$ bounds a disc not containing $c'_0$, that we call $D_0$. The circle $c'_0$ bounds a disc containing $c_0$, that we call $D'_0$. And they both bound an open annulus called $A_0$.  Analogously we define $D_1$, $D' _1$ and $A_1$.
\item  The orbit $O(x)$ of a point $x$ in $S^2\times\set{-1}$, crosses  $S^2\times\set{1}$ if and only if $x\in A_0$ and $O(x)\cap S^2\times\set{1}\in A_1$,

\item There is a well defined crossing map $P:A_0\to A_1$. Consider the radial foliation $V_0$ in $A_0$  Then the image of a radial foliation  under $P$ intersect transversally a radial foliation  in $A_1$ and it extends to a foliation in $A_1\cup c_1 \cup c'_1$.
 \end{itemize}
The complement of $S^2\times{[-1,1]}$ in $S^3$ are 2 balls, one in the basin of attraction of a source $r$ (that has $S^2\times\set{-1}$ in the boundary ), and the other in the basin of attraction of a sink $a$.
\end{theo*}

We consider the set $K=\set{(x,y) \,\,\,\text{ tq }\,\,\abs{y}\leq 1\text{ and  } 0\leq x\leq 1}$, and in this set, a flow $\phi_0$ of a vector field $Y_0$ in $\RR^2$ with the following properties:\begin{itemize}
                                        \item The vector field $Y_0$ is Morse-Smale with a source  $p_1=(0,1/2)$, a sink  $p_2=(0,-1/2)$ and a  saddle  $p_3=(1/2,0)$.
                                        \item The vectirfield simetric (in $\RR^2$) with respect to $x$ in a neighborhood of the interval $\set{(0,y) \,\,\,-1\leq y\leq 1}$
                                        \item The saddle $p_3$ is such that  a branch of the unstable manifold intersects the basin of the sink,the other  intersects a corner of $K$.
                                         \item The saddle $p_3$ is such that  a branch of the stable manifold intersects the basin of the source,the other  intersects a corner of $K$.
                                      \end{itemize}

We take a point $q =(1,1-\epsilon)$  for some positive and small $\epsilon$, and  another point of the orbit of $q$ that we call $q'$ with $x$ coordinate $1$. 
 We call $K'$ to the "square" delimited by
 \begin{itemize}
 \item the segments $\set{(0,y)\,\,\,tq\,\,\abs{y}\leq 1}$,
 \item the segment $\set{(x,1)\,\,\,tq\,\,0\leq x \leq 1}$,
 \item  the segments $\set{(1,y)\,\,\,tq\,\,1-\epsilon\leq y\leq 1}$,
 \item the vertical segment that joints $q'$ with $(1,-1)$,
 \item  the orbit segment joining $q$ and $q'$,
 \item  the segment $\set{(x,-1)\,\,\,tq\,\,0\leq x \leq 1}$.
 \end{itemize}
The set $K'$ is diffeomorphic to  $$C=\set{(x,y)\,\,\tq \,\, 0\leq x\leq 1 \,\,\text{ y } \, \abs{y}\leq 1}\,$$ where the orbit of $q$ is sent to the segment $[1,y]$ and out of a small neighborhood of this segment everything is left unchanged.
\begin{figure}[htb]
\begin{center}
\includegraphics[width=0.48\linewidth]{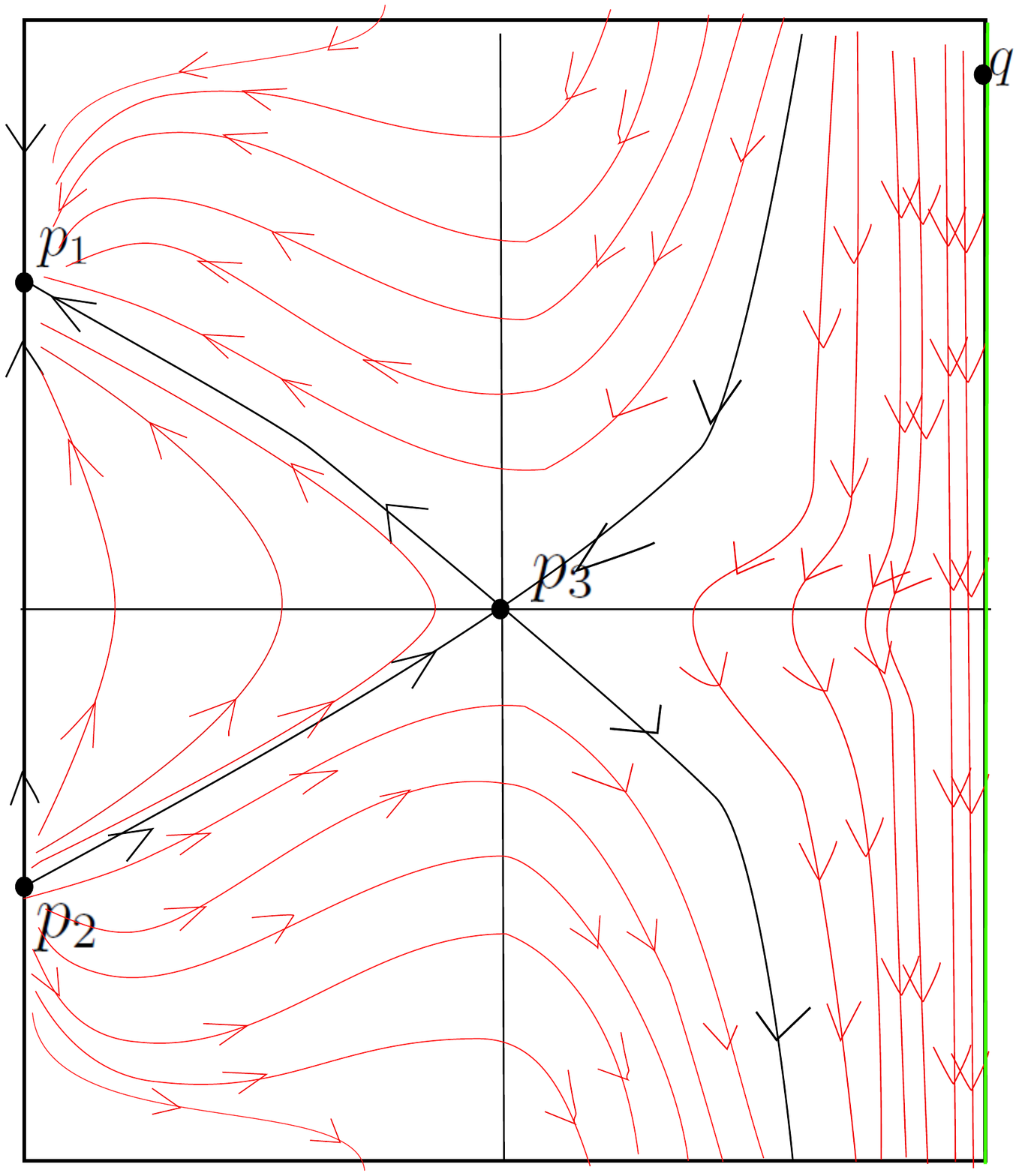}
\end{center}
\end{figure}

 We call $Y_1$ to the vector field tangent to the flow $\phi_1$ obtained by composing the flow with the diffeomorphism above.

We define a $C^{\infty}$ function $f:\RR\to\RR$ such that:
  \begin{itemize}
    \item $f(0)=0 $
    \item $f(1/2)=1$
    \item $f$ is decreasing  $(1/2,1)$
    \item $f'(x)\neq0$ in $[1/2,1]$
    \item $f(1)=0$.
   \end{itemize}

Now we consider  $C\times S^1$ where if $(x,y,\theta)\in C\times S^1$ then $(x,y)\in C $ with $x\in[0,1]$ and $y\in[-1,1]$.  In $S^1 $ we consider the vector field in $C\times S^1$
that if $$v\in T(x,y)C\,\,\,\text{ then } \chi^{+}=(Y_1(v),f(x))\,.$$ 
We can also consider in another copy  of $C\times S^1$ the vector field $$\chi^{-}=(Y_1(v),-f(x))\,.$$ 
 We call $p'_1$, $p'_2$ and $p'_3$ to the source, the sink and the saddle for $\chi^{-}$.
t

\begin{figure}[htb]
\begin{center}
\includegraphics[width=0.48\linewidth]{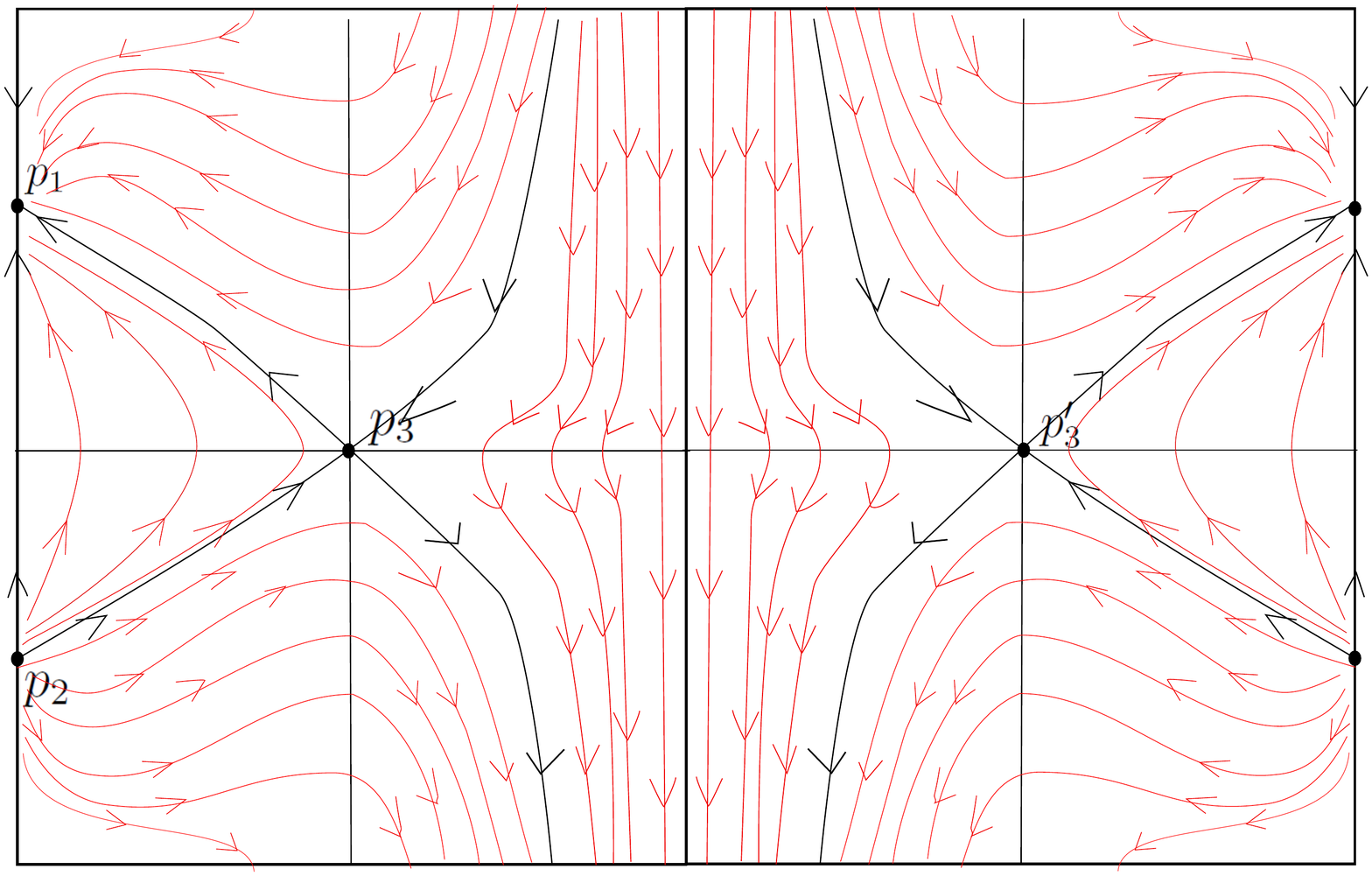}
\end{center}
\end{figure}
We re write  $C\times S^1$ in Cartesian coordinates as $\mathbb{D}^2\times [-1,1]$ so that the vector field is entering at $\mathbb{D}^2\times\set{1}$, points out at  $\mathbb{D}^2\times\set{-1}$ and is tangent to  $(\partial(\mathbb{D}^2)) \times{[-1,1]}$. We can do this change of coordinates and still have a differentiable vector field because of the symmetry properties of $Y_0$
We paste $2$ copies of  $\mathbb{D}^2\times [-1,1]$ along  $(\partial(\mathbb{D}^2)) \times{[-1,1]}$.
In one copy we have $\chi^{+}$ and in the other we have $\chi^{-}$. Since the vector fields are equal in  $\partial(\mathbb{D}^2) \times{[-1,1]}$, both are $C^{\infty}$ even restricted to the boundary and no orbit crosses  $\partial(\mathbb{D}^2) \times{[-1,1]}$, we can define a gluing map such that  the resulting vector field  $\chi$  defined in  $S^2\times [-1,1]$ is smooth.

\begin{figure}[htb]
\begin{center}
\includegraphics[width=0.48\linewidth]{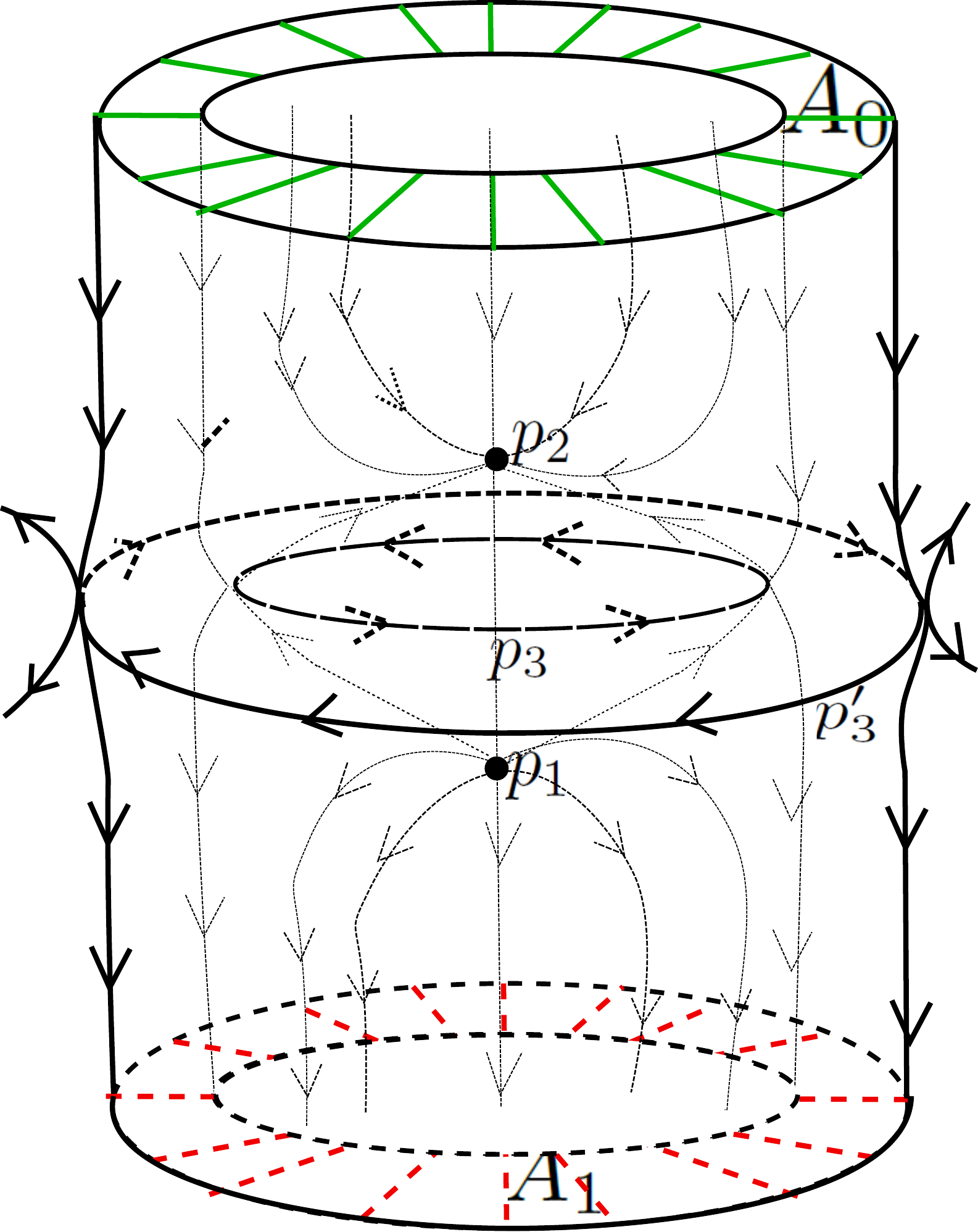}
\end{center}
\end{figure}

For our convenience, the intersections of the stable and unstable manifolds of the saddle periodic orbits with  $\partial(S^2\times [-1,1])$ are named as follows.
 \begin{itemize}
  \item $W^{s}(p_3)\cap S^2\times{[-1,1]} = c_0$ in $S^2\times{-1}$,
  \item $W^{s}(p'_3)\cap S^2\times{[-1,1]} = c'_0$ in $S^2\times{-1}$,
\item $W^{u}(p_3)\cap S^2\times{[-1,1]} = c_1$ in $S^2\times{1}$.
\item  $W^{u}(p'_3)\cap S^2\times{[-1,1]} = c'_1$ in $S^2\times{1}$.
 \item The intersection of the invariant manifolds of the saddles  with the boundary of $S^2\times{[-1,1]}$ are the disjoint circles, $c_0$,  $c'_0$,  $c_1$, $c'_1$.
 \item The circle $c_0$ bounds a disc not containing $c'_0$, that we call $D_0$. The circle $c'_0$ bounds a disc not containing $c_0$, that we call $D'_0$. And they both bound an annulus called $A_0$  Analogously we define $D_1$, $D' _1$ and $A_1$.
 \end{itemize}
To complete the flow to $S^3$ we add a ball contained in the basin of attraction of a a sink $a$ and a ball contained in the basin of repulsion of a source $r$ in the remaining space . We can do this since  the vector field is entering at  $S^2\times{1}$, and vector field is pointing out at  $S^2\times{-1}$. 

 This way, the orbits  behave in one of the following ways (see figure \ref{compflow.f}) :
 \begin{enumerate}
   \item They go from  from $r$ to $a$ and cross the annuli $A_0$ and $A_1$.
  \item They go to the periodic orbit $p'_3$  for the future or the past.
 \item They go for the future to  $p_2$ or $p_2'$,
\item They go for the past to  $p_1$ or  $p_1'$ 
 \end{enumerate}

\begin{figure}[htb]
\begin{center}
\includegraphics[width=0.55\linewidth]{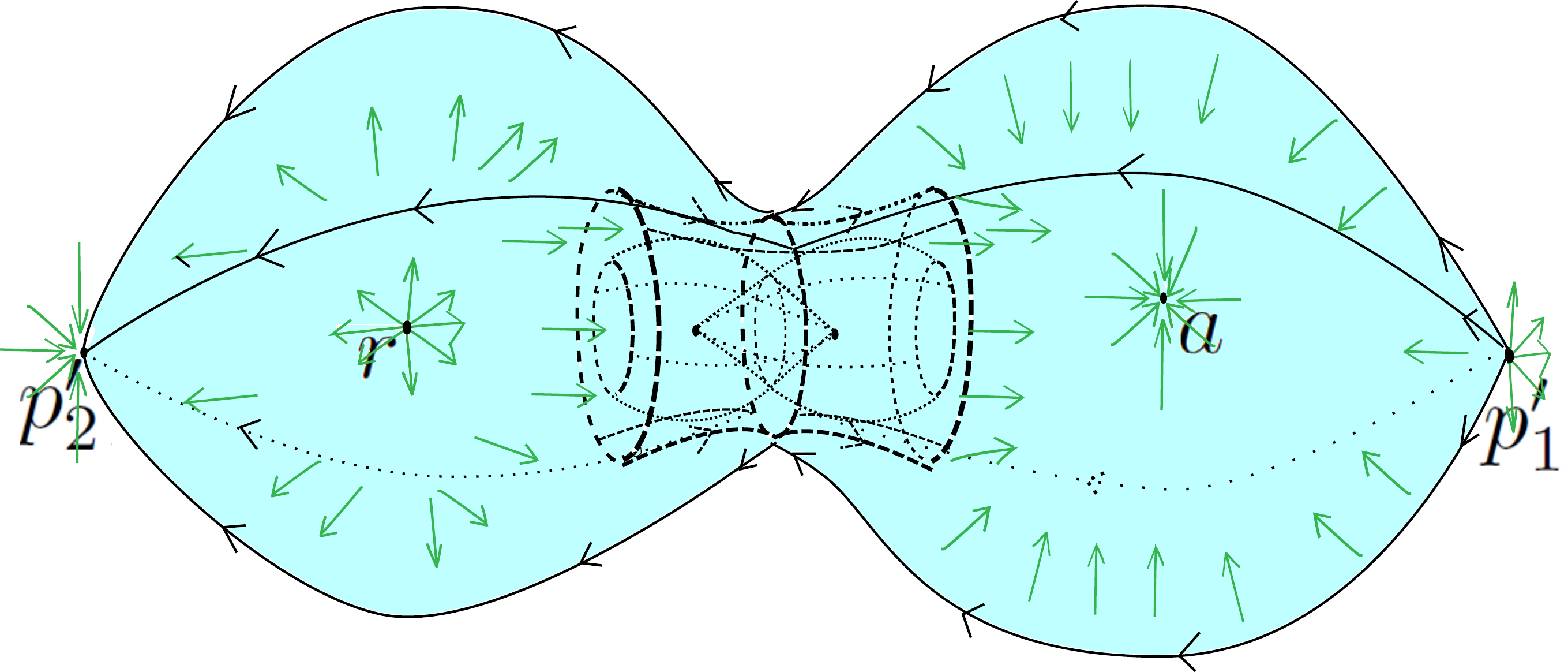}\label{compflow.f}
\end{center}
\caption{The vector field $\chi$ in $S^3$.}
\end{figure}

The next lemma is to check all the conditions of theorem \ref{t.tube}, except for the transversality condition, that we will check in the next subsection
\begin{lemm}\label{t.tube1}
The vector field $\chi$ defined in $S^2\times [-1,1]$ has the following properties:
\begin{itemize}
\item The  vector field $\chi$ is such that the chain recurrent set consists of 2 sources $p_1$ and  $p_1'$ ,  2 sinks $p_2$ and $p_2'$, and  2 periodic saddles $p_3$ and $p_3'$.

\item  The orbit $O(x)$ of a point $x\in S^2\times\set{-1}$ crosses  $S^2\times\set{1}$ if and only if $x\in A_0$ and $O(x)\cap S^2\times\set{1}\in A_1$
\end{itemize}
\end{lemm}
\proof
The flow $\phi_{Y_1}$ is such that the only chain  recurrent points are the sinks or sources. Therefore all points except for $p_1$,  $p_1'$,   $p_2$, $p_2'$, $p_3$ and $p_3'$ are non-chain recurrent  for  $\chi^{+}$ and  $\chi^{-}$.
For  $\chi^{+}$ over $C\times S^1$   the points of the form $(0,y,\theta)$ have zero $S^1$ coordinate and so $p_1$ and $p_2$ are singularities. Moreover, the eigenvalues of the differential of the vector field at  $p_2$ are one real and negative and one complex with negative real part.
The same reasoning gives us that $p_1$ is a source and for $\chi^{-}$ that $p_1'$ is a source and $p_2'$ is a sink.
For  $\chi^{+}$ the point $p_3$ is periodic with its orbit tangent to the $S^1$ coordinate, therefore since for $Y_1$ the point $p_3$ was a saddle, it is a saddle for $\chi^{+}$  as well. The same reasoning gives us that $p_3'$ is a saddle periodic orbit for $\chi^{-}$.
So for the vector field $\chi$,
by construction, there are no orbits crossing from one copy of  $\mathbb{D}^2\times [-1,1]$ to the other therefore we have that  $\chi$ is such that the chain recurrent set consists of 2 sources $p_1$ and  $p_1'$ ,  2 sinks $p_2$ and $p_2'$, and  2 periodic saddles $p_3$ and $p_3'$.

The second item comes from the fact that all orbits of the points between the stable manifold of $p_3$ and $q$ in the segment $[1,y] $ cross again $[1,y]$ between the unstable manifold of $p_3$ and $q'$ for the flow  $\phi_{Y_0}$.  
These are the half open segments that (up to a diffeomorphism) are the base of half of   $A_0$  and  $A_1$ in the product corresponding to  $\chi^{+}$. The other half of   $A_0$  and  $A_1$ that correspond to  $\chi^{-}$ behaves in the same way. 
The other orbits of the vector field   $\phi_{Y_0}$. are in the basin of attraction of the sink, or the basin of repulsion of the source or are the singularities themselves. This is also true in the base of the product dynamics of  $\chi^{+}$  and therefore for  $\chi^{+}$ . The dynamics for $\chi^{-}$  is the same.

\endproof
\subsubsection{A radial foliation and the crossing map $P$}

The aim of this subsection is to prove the last part of Theorem \ref{t.tube}.

Since every orbit of the points in $A_0$ cuts $A_1$ at some moment, we define the first return  map $P:A_0\to A_1$. We take polar coordinates in $A_0$ and $A_1$ (that is, we take coordinates in $S^1\times(0,1)$). We write the diffeomorphism $P$ in this coordinates as $$P(\theta,r)=(P_{\theta}(\theta,r),P_{r}(\theta,r))$$
\begin{figure}[htb]
\begin{center}
\includegraphics[width=0.55\linewidth]{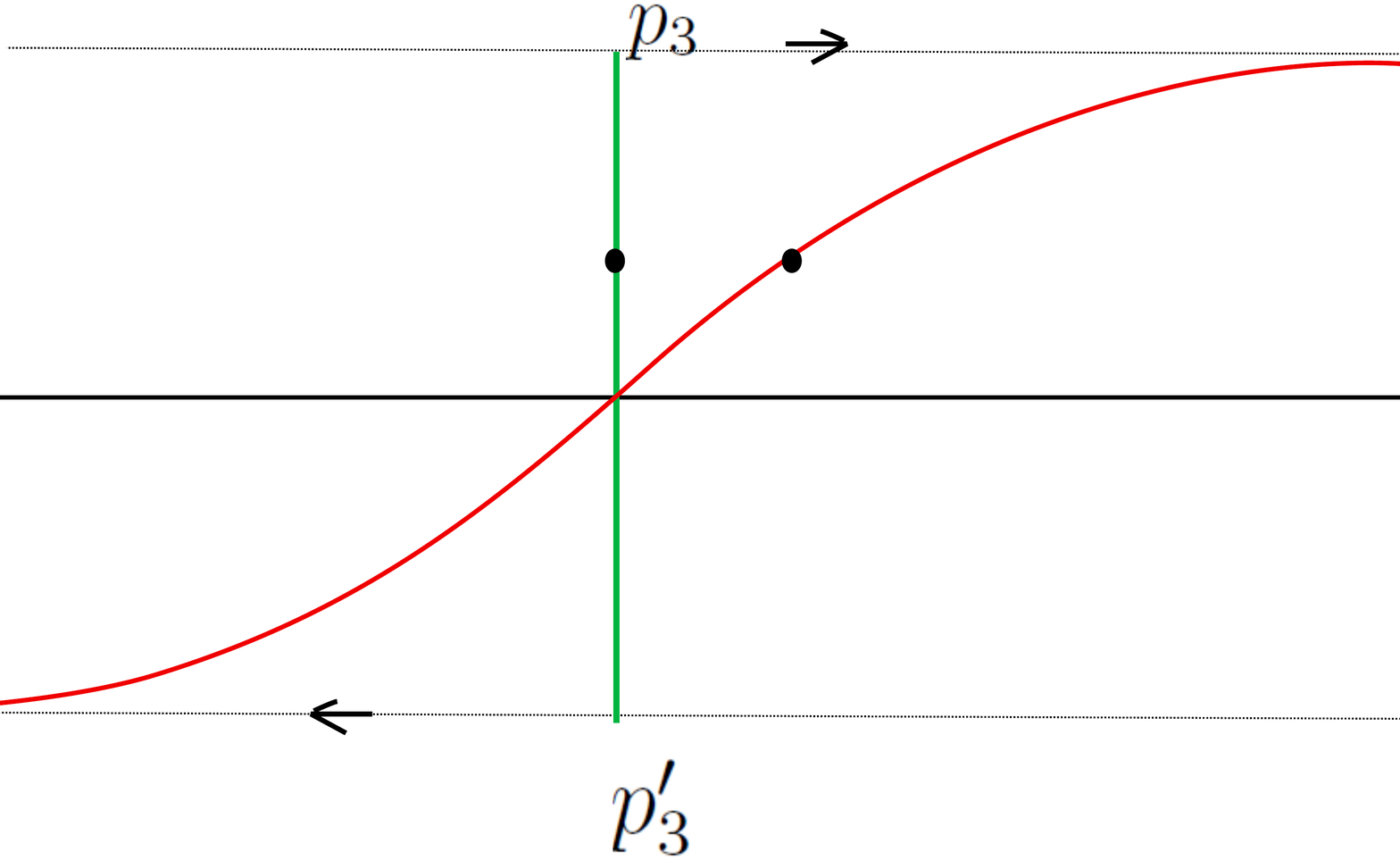}
\end{center}
\caption{The lift of the  map $P:A_0\to A_1$ to  $\RR\times(0,1)$}
\end{figure}

\begin{lemm}\label{l.transvers}
Let us write  $A_0$ and $A_1$ in coordinates $S^1\times(0,1)$ and
let $P:A_0\to A_1$ be the first return map from $A_0$ to $A_1$, $P(\theta,r)=(P_{\theta}(\theta,r),P_{r}(\theta,r))$, defined by the vector field $\chi$.
If $D P(0,r)=(w,z)$ then $w\neq 0$ for all $r \in (0,1)$.
As a consequence, the image of the segment $(0,r)$ with $r \in (0,1)$,   cuts transversally any segment of the form $(0,r)$ with $r \in (0,1)$ in $A_1$.

\end{lemm}

\proof
 From the construction of $\chi$ there is a circle $ S^1\times q \in A_0$ where   $\chi^{+}$ and  $\chi^{-}$ coincide. We consider new coordinates in  $A_0$ and $A_1$ so that the coordinates of $q $ are now  $(\theta,1/2)$.

Let us first consider a lift of $A_0$, to the strip  $\widehat{A_0}$ parametrized by $\RR\times(0,1)$, we also take the lift of $A_1$, to $\widehat{A_1}$, and the lift of $P$, $\widehat{P}$.
We we may choose  this lifts so that the rotation in the sense of $p_3$ is lifted  as positive.

Let us take a point $x=(\theta,r)$ in  $A_0$. We define the  time that it takes for $x$ to reach $A_1 $ under the flow of  $\chi$ as $T_x$.

Suppose that $r<1/2$ Recall that the vector field  is $\chi=(Y_1,Y_2)$ were  $$Y_2(\theta)=f(g(r)) \;d\theta\,,$$ and  therefore $$\frac{\partial P_{\theta}(\theta,r)}{\partial r}=T_x f(g(r))'g(r)'\,,$$
 and therefore non vanishing.
 Suppose that $r>1/2$,  then the vector field $Y_3$ is defined as $$Y_3(\theta)=-f(g(r)) \;d\theta\,,$$ and  therefore $$\frac{\partial P_{\theta}(\theta,r)}{\partial r}=T_xf(g(r))'g(r)'\,,$$
 and therefore non vanishing.

 At $r=1/2$ since the lateral derivatives are not 0 and the function is smooth then.

 $$\frac{\partial P_{\theta}(\theta,r)}{\partial r}\neq0\,.$$

Note that from the construction of the flow $T_x$ goes to infinity as $r$ goes to $0$ or $1$
\endproof

\end{document}